\documentclass[a4paper]{article}
\usepackage[margin=3cm]{geometry}
\usepackage[utf8]{inputenc}
\usepackage[T1]{fontenc}
\usepackage{lmodern,amssymb,amsfonts,dsfont,mathrsfs,nccmath}
\usepackage{mathbbol}
\usepackage{xcolor}
\usepackage{bm}
\usepackage[colorlinks,linkcolor=blue!70!black,citecolor=magenta!90!black, psdextra]{hyperref}
\usepackage{doi}
\usepackage{natbib}
\RequirePackage[amsthm, amsmath, thmmarks, hyperref, framed]{ntheorem}
\usepackage[nameinlink]{cleveref} 
\usepackage{graphicx}
\usepackage{booktabs}
\usepackage{caption} 
\usepackage{tikz}
\usepackage{float}
\usepackage{mathtools}
\usepackage{comment}
\usepackage{microtype}

\makeatletter
\newcommand{\mytag}[2]{%
  \text{#1}%
  \@bsphack
  \begingroup
    \@onelevel@sanitize\@currentlabelname
    \edef\@currentlabelname{%
      \expandafter\strip@period\@currentlabelname\relax.\relax\@@@%
    }%
    \protected@write\@auxout{}{%
      \string\newlabel{#2}{%
        {#1}%
        {\thepage}%
        {\@currentlabelname}%
        {\@currentHref}{}%
      }%
    }%
  \endgroup
  \@esphack
}
\makeatother
\numberwithin{equation}{section}
\usepackage[shortlabels]{enumitem}
\newlist{mathlist}{enumerate}{5}
\setlist[mathlist]{label={\textup{(\roman*)}}}
\theoremstyle{plain}
\newtheorem{thm}{Theorem}[section]
\newtheorem{prop}[thm]{Proposition}
\newtheorem{lem}[thm]{Lemma}

\theoremstyle{definition}
\newtheorem{hypo}{Assumption}

\newtheorem{defn}[thm]{Definition}

\theoremstyle{remark}
\newtheorem{rem}[thm]{Remark}
\allowdisplaybreaks[1]

\usepackage{setspace}
\makeatletter
\newcommand{\oset}[3][0.1ex]{%
  \mathrel{\mathop{#3}\limits^{
    \vbox to#1{\kern-2\ex@
    \hbox{$\scriptstyle#2$}\vss}}}}
\makeatother

\renewcommand{\epsilon}{\varepsilon}

\renewcommand{\rho}{\varrho}

\newcommand{\1}{\mathds{1}}

\newcommand{\ie}{\textit{i.e.}\ }
\newcommand{\eg}{\textit{e.g.}\ }

\newcommand{\anc}[1]{\overset\leftarrow{#1}}
\DeclareMathOperator{\R}{\mathbb{R}}
\DeclareMathOperator{\e}{\mathrm{e}}
\DeclareMathOperator{\E}{\mathbb{E}}
\DeclareMathOperator{\N}{\mathbb{N}}

\title{\large{\textbf{MOMENTS OF DENSITY-DEPENDENT BRANCHING PROCESSES AND THEIR GENEALOGY}}}

\usepackage{authblk}

\author[1,2,3]{Mathilde \textsc{André}} 
\author[4]{Félix \textsc{Foutel-Rodier}}
\author[2]{Emmanuel \textsc{Schertzer}}

\affil[1]{Institut de Biologie de l'ENS (IBENS), École Normale Supérieure, PSL Université, CNRS UMR 8197,
INSERM U1024, Paris, France}
\affil[2]{Faculty of Mathematics, University of Vienna, Oskar-Morgenstern-Platz 1, 1090 Wien, Austria}
\affil[3]{Marie et Louis Pasteur University, CNRS, LmB (UMR 6623), F-25000 Besançon, France}
\affil[4]{Université Paris Cité, CNRS, MAP5, F-75006 Paris, France}

\begin{document}

\maketitle

\begin{abstract}
A density-dependent branching process is a particle system in which
individuals reproduce independently, but in a way that depends on the
current population size. This feature can model a wide range of
ecological interactions at the cost of breaking the branching property.
We propose a general approach for studying the genealogy of these models
based on moments. Building on a recent work of Bansaye, we show how to
compute recursively these moments in a similar spirit to the many-to-few
formula in the theory of branching processes. These formulas enable one
to deduce the convergence of the genealogy by studying the population
density, for which stochastic calculus techniques are available. As a
first application of these ideas, we consider a density-dependent
branching process started close to a stable equilibrium of the ecological
dynamics. We show that, under a finite second moment assumption, its
genealogy converges to Kingman's coalescent when the carrying capacity of
the population goes to infinity.
\end{abstract}

\setcounter{tocdepth}{2}\tableofcontents
\section{Introduction}

\subsection{Motivation}

A genealogy is a forest structure that represents the degree of kinship
between individuals in a population. It is a central notion in population
genetics, in particular because genealogies are instrumental in studying
patterns of neutral genetic diversity \citep{wakeley_coalescent_2009,
tellier2014coalescence}. The mathematical theory of genealogies has been
mostly developed around the notion of coalescents, following the
seminal work of \cite{Kingman:1982aa} and its
subsequent generalizations \citep{Pitman:1999aa,
sagitov1999general, donnelly1999particle, schweinsberg2001coalescents}. 
A coalescent is a stochastic process obtained by sampling a finite number
of individuals at some reference generation, and following backward in
time their ancestral lineages until they merge (coalesce). This procedure
constructs a random forest which is usually encoded as a stochastic process with
values in the partitions.

Since their introduction, coalescents have been shown to describe the
limit of the genealogy of a wide variety of population models 
\citep{mohle_1998, mohle1999, Mohle_Sagitov, kaj_coalescent_2003,
schweinsberg2003coalescent}. However, a feature shared by almost all of
these models is that population size is fixed or, if individuals can be
of different types, that the number of individuals of each type is given. The
archetypical example of such models is the Wright--Fisher model, and its
generalizations to multiple types or to Cannings models
\citep{cannings1974latent, Etheridge:2011vm}. Despite their importance,
these processes do not model appropriately ecological interactions as
well as many important features of biological populations. For instance,
they can hardly model individuals spread across a spatial continuum or
with a continuous age structure. Conversely, density-dependent branching processes
is another well-studied class of models which seems suited to model such
phenomena. In these models, individuals reproduce independently, but in a
way that depends on the state of the population and which can model biological
interactions \citep{Kurtz:1978aa, Fournier2004,
champagnat2006microscopic, meleard-tran, etheridge2023looking}. The
logistic branching process \citep{Lambert2005} or the Bolker--Paccala
model \citep{bolker1997using} are typical examples of density-dependent
branching processes. This provides motivation for extending our knowledge
of genealogies to these processes, beyond models with fixed population
structure.

Unfortunately, the standard approach to proving convergence towards a
coalescent relies heavily on the property that the population structure
is fixed. This constraint entails that one can define a coalescent from
the discrete model as a Markov process. One can then rely on the numerous
techniques available for Markov processes to study it. The main objective of
our work is to propose some new ideas that can circumvent this issue. As
we will explain in more details below, instead of following ancestral
lineages backward in time, we will decompose forward in time the
genealogy at its last branch point. This approach is very close in spirit
to the spinal methods which have been used extensively in branching
process theory \citep{Lyons_Pemantle_Peres, Bansaye2011, Harris_Roberts}
and builds on the recent work of \cite{Bansaye} in that direction.

As a first application of this approach, we study the genealogy of a
general class of density-dependent branching processes with no types. 
Doing so, the main focus is to introduce the
methodology, in order to extend it to
structured population models in subsequent works. 
We show that, if the population starts close to a stable equilibrium of the
ecological dynamics and if the reproduction law has finite variance, the
genealogy of the population converges to Kingman's coalescent at large
scales.

\subsection{Model and main result}\label{sec:model}

\subsubsection*{Model} 
Fix some scaling parameter $K > 0$. The parameter $K$
corresponds to the carrying capacity of the model, which gives the
typical number of individuals in the population. For each $z > 0$,
consider some $q(z) > 0$ and some random variable $L(z)$ in
$\{0, 1, 2, \ldots\}$. They correspond respectively to the branching rate
and to the number of offspring of an individual when the population
density (that is, the size of the population divided by the carrying
capacity) is $z$. We start the population from $Z_0$ particles at time
$t=0$. Then, if the current population size is $n$, each individual dies
independently at rate $q(n/K)$, upon which it is replaced by an
independent number of offspring distributed as $L(n/K)$. This constructs
our model iteratively, until a potential explosion time.

More formally, we envision each individual as the vertex of a graph and,
attaching parents to offspring, we construct the population as a random 
 forest in which each vertex is further endowed with a birth and a death time.
We index individuals as elements of 
 $\mathcal{U}=\bigcup_{n\in\N}\N^n$, the set of Ulam--Harris labels, with the usual
interpretation that $ui$ is the $i$-th child of individual $u \in
\mathcal{U}$. Note that, since our population is a forest rather than a
tree, it does not start from a unique root vertex $\varnothing$ but from
a finite number $Z_0$ of roots labeled $\{1,\ldots, Z_0\}$ that form the
first generation. The procedure described above endows each individual
$u$ with a birth time $\lambda_u$ and a death time $\theta_u$. (With the
convention that $\lambda_i = 0$ for each root $i \in \{1,\ldots, Z_0\}$.)
We will denote by 
\[
    \mathcal{N}_t\ =\ \{ u \in \mathcal{U} : \lambda_u \leqslant t < \theta_u \}
\]
the set of individuals alive at time $t$, and by $Z_t =
|\mathcal{N}_t|$ the population size at that time. Finally, for two
individuals $u, v \in \mathcal{U}$, we let $u \wedge v$ denote their
most-recent common ancestor. The genealogical distance between two
individuals at time $t$ can then be defined as
\begin{equation}\label{eq:gen_distance}
    \forall u, v \in \mathcal{N}_t,\quad D_t(u,v)\ =\ t - \theta_{u \wedge
    v}.
\end{equation}
Note that $D_t(u,v)$ also corresponds to the coalescence time between $u$
and $v$.
\subsubsection*{Assumptions}

Let $m(z) = \E\big[L(z)\big]$ be the mean
number of offspring of an individual at population density $z$. When $m$
is finite, a result of \cite{Kurtz:1978aa} (see also
\citet[Chapter~11]{ethier_1986}) shows that the population density $(Z_t
/ K)_{t \geqslant 0}$ converges as $K \to \infty$ to a deterministic limit
$(z_t)_{t \geqslant 0}$ which is the solution to the nonlinear differential
equation 
\[
    \dot{z}_t \ =\ z_t q(z_t) (m(z_t)-1).
\]
\cite{Kurtz:1978aa} also obtained that under an additional second moment
assumption at stable equilibrium, the deviations converge weakly to an
Ornstein--Uhlenbeck process with gaussian stationary distribution.

Our first assumption ensures that the limiting dynamics has a stable
equilibrium, which we set to $1$ without loss of generality.

\begin{hypo}[Stable equilibrium]\label{hyp:equilibrium}
    The mean number of offspring $m(z) \coloneqq \E[L(z)]$ is finite for any $z >
    0$, and $m(1) = 1$. Moreover, $q$ is continuous at $z=1$ and $m$ is
    continuously differentiable at $z=1$ with $m'(1) < 0$.
\end{hypo}

Kingman's coalescent corresponds to the scaling limit of the genealogy of
populations in which the variance of the offspring size is finite. 
If this is relaxed, an individual can have a number of offspring of the
same order as the population size, leading to more general
$\Lambda$-coalescents \citep{schweinsberg2003coalescent, Mohle_Sagitov}.
Accordingly, our second assumption requires that the number of offspring
close to the equilibrium has a finite second moment.

\begin{hypo}[Finite second moment]\label{hyp:moments} 
    The map $z \longmapsto m_2(z) \coloneqq \E[L(z)(L(z)-1)]$ is
    continuous at $z = 1$. Moreover, there exists $\varepsilon > 0$ such
    that 
    \begin{equation} \label{eq:UI_moment}
        \adjustlimits\limsup_{A \to \infty}\sup_{z \in (1-\varepsilon, 1+\varepsilon)} 
        \E[L(z)^2 \1_{\{ L(z) > A\}}] = 0.
    \end{equation}
\end{hypo}

Note that the uniform integrability condition \eqref{eq:UI_moment} is
satisfied if $\sup_{z \in (1-\varepsilon, 1+\varepsilon)} \E[L(z)^{2+\eta}] <
\infty$ for some $\eta > 0$.

\subsubsection*{Main result} Our main result is that the genealogy of the
population, when started close to the stable equilibrium, converges to
Kingman's coalescent. It can be formulated in several ways, and we choose
to use the standard formalism of coalescent theory. In this
direction, fix $t > 0$ and let $(U_1,\ldots, U_k)$ be $k\geqslant 1$ individuals
sampled uniformly without replacement from the population at time $Kt$.
(Under our assumptions there will be more than $k$ individuals in the
population with high probability.) Let us define a process
$(\Pi^{k,K}_\tau)_{0 \leqslant \tau \leqslant t}$ with values in the partitions of
$\{1,\ldots, k\}$ by
\[
    i \sim_{\Pi^{k,K}_\tau} j \iff D_{Kt}(U_i, U_j) \leqslant \tau K.
\]
This process encodes the subforest spanned by $(U_i)_{i \leqslant k}$ as a
coalescent. Let us also recall that the $k$-Kingman coalescent with
rate $r > 0$ is defined as the stochastic process $(\Pi^{k}_\tau)_{\tau
\leqslant t}$ with values in the partitions of $\{1,\ldots,k\}$ such that
$\Pi^{k}_0$ is the partition of $\{1,\ldots,k\}$ into singletons, and
such that any two blocks merge at rate $r$.

\begin{thm}[Convergence of the genealogies]\label{thm:main}
    Suppose that Assumption~\ref{hyp:equilibrium} and Assumption~\ref{hyp:moments}
    hold and let $t > 0$, $k\geqslant 1$.
    \begin{enumerate}
    \item[(i)] The following holds in probability
        \[
            \adjustlimits\lim_{K \to \infty} \sup_{0 \leqslant s \leqslant Kt}\ 
            \Big| \frac{Z_s}{K} - 1 \Big|\ =\ 0.
        \]
    \item[(ii)] Conditional on the event $\{Z_{Kt}\geqslant k\}$, let
        $(\Pi^{k,K}_\tau)_{\tau \leqslant t}$ be the coalescent obtained from
        a uniform sample of size $k$ at time $Kt$. Then
    \[
        \lim_{K \to \infty}\ (\Pi^{k,K}_\tau)_{\tau \leqslant t}
       \ =\ (\Pi^{k}_\tau)_{\tau \leqslant t}
    \]
    in distribution for Skorohod's topology, where $(\Pi^k_\tau)_{\tau 
    \geqslant 0}$ is a $k$-Kingman coalescent with rate $q(1)m_2(1)$. 
    \end{enumerate}
\end{thm}

\subsubsection*{Related literature}
Recently, \cite{forien2025} considered a branching process with logistic
competition that falls within our assumptions. The author considers the
scaling limit of the frequency process of the population, obtained by
giving different labels to the initial individuals and following their
frequencies as time goes on. One of the results derived in
\cite{forien2025} is that, under a ($2+\eta$)-moment assumption similar (but weaker) to
Assumption~\ref{hyp:moments}, the frequency process converges to a
Fleming--Viot process. The latter is well known to be in moment duality
with Kingman's coalescent (see \eg \cite{ethier_flemingviot_1993}), which
is consistent with our main result. 
A similar result was previously obtained by \cite{billiard2015stochastic}
in an adaptive dynamics model with binary branching and logistic
competition, that can also encompass the co-existence of several types.
The corresponding genealogy is described in \cite{lepers2021inference},
but the convergence to it is not formally derived.
We would also like to mention that the results in \cite{forien2025} go
beyond the case of finite variance reproduction law, and also give the
limit of the frequency process when the reproduction law has a heavy
tail.

Near the completion of our work, we became aware that \cite{garrett2025}
derived a result very similar to the theorem above. Building on the
estimates of \cite{forien2025}, they prove convergence of the genealogy of
the branching process with logistic competition to Kingman's coalescent.
Besides, they also derive the corresponding convergence to a
$\mathrm{Beta}(2-\alpha, \alpha)$ coalescent -- with $\alpha \in [1,2)$ --
when the reproduction law has a heavy tail. The main ingredient used in 
\cite{garrett2025} is a lookdown construction of the process
\citep{donnelly1999particle}, which is very different from our techniques. The
lookdown process leads to a very efficient proof of the convergence of
the genealogy. 
Again, our results do not rely on those of \cite{forien2025} but we
re-establish them in a more general setting instead. In particular, we relax the
$2+\eta$ moment assumption on the offspring distribution and extend the
interaction beyond the logistic case. This higher level of generality
requires additional technical work (see Remark~\ref{rem:technicalities}).

Adapting the methodology of \cite{garrett2025} to a population with a general trait
(or type) structure would require to use continuous levels in a similar way
to \cite{Kurtz2011, Etheridge2014}, which seems challenging. Conversely,
although our approach might seem more indirect, spinal methods have
already proved to be very successful when studying the genealogy of
branching processes with various trait structures
\citep{harris_coalescent_2020,
foutelrodier2024momentapproachconvergencespatial, schertzer2023spectral,
FST_Semipushed_2024, harris_universality_2024, harris_coalescent_2024,
boenkost2022genealogy, Johnston2019}. The skeleton of our proof should,
in principle, apply generally. We hope (and believe) that our approach
should scale to more complicated population models, in particular to
models incorporating space. 

\subsection{Proof heuristics and outline}

Showing the convergence of the $k$-coalescent would involve computing the
limit of 
\[
    \E\bigg[ 
        \frac{1}{(Z_t)_k} 
        \sum_{\bm{u} \in \mathcal{N}_t:u_1\neq \cdots\neq u_k} 
        \varphi\big(D_t(\bm{u})\big) \bigg],
\]
for any continuous bounded $\varphi$ and where $D_t(\bm{u}) =
(D_t(u_i, u_j))_{i,j \leqslant k}$ is the matrix of pairwise coalescence times
of a sample $\bm{u} = (u_i)_{i \leqslant k}$, and $(X)_k = X(X-1)\cdots(X-k+1)$
is the $k$-th descending factorial of $X$. Since the population size
remains close to its equilibrium, this expression can be approximated by
\begin{equation}\label{eq:moment_UMS_intro} 
    \frac{1}{(Z_0)_k} 
    \E\bigg[\e^{\beta (Z_0-Z_t)} 
        \sum_{\bm{u} \in \mathcal{N}_t:u_1\neq \cdots\neq u_k} 
        \varphi\big(D_t(\bm{u})\big)
    \bigg],
\end{equation}
for any $\beta \geqslant 0$. We call expressions of the form
\eqref{eq:moment_UMS_intro} penalized $k$-moments of the population. 

\begin{rem}\label{rem:beta}
    Estimating \eqref{eq:moment_UMS_intro} with $\beta = 0$ is already
    the key to the approach proposed in \cite{Foutel_Schertzer} to
    studying the genealogy of branching processes. However,
    \eqref{eq:moment_UMS_intro} might be infinite for $\beta=0$ if the
    offspring distribution lacks higher order moments, which requires a
    cumbersome truncation of the offspring law as in \cite{harris_coalescent_2020, 
    boenkost2022genealogy}. The exponential penalization for $\beta > 0$
    circumvents this issue elegantly, which is an idea borrowed from
    \cite{harris_coalescent_2024}.
\end{rem}

Recently, \cite{Bansaye} provided a Feynman--Kac formula for
\eqref{eq:moment_UMS_intro} when $k=1$. It expresses
\eqref{eq:moment_UMS_intro} in terms of a density-dependent branching
process with an additional immigration from a distinguished line of
individuals called a \emph{spine}. The latter is obtained by a martingale
change of measure of the original process, which is reminiscent of the
usual many-to-one formula for branching processes.

In Section~\ref{sec:Topology}, we extend this expression to any $k \geqslant 1$
in a similar spirit to the many-to-few formula \citep{Harris_Roberts}.
Namely, we show in Theorem~\ref{thm:recursion_moment} that \eqref{eq:moment_UMS_intro} 
can be computed recursively from lower order moments using a
density-dependent branching process with immigration along $k$ different
spines. This result is the cornerstone of our approach and is of
independent interest.

The bulk of the technical work that remains is to show that the
population indeed remains close to its equilibrium and to estimate the
Feynman--Kac term arising from Theorem~\ref{thm:recursion_moment}. This is
carried out in Section~\ref{sec:calcul_sto} using stochastic calculus. Finally,
the limit of \eqref{eq:moment_UMS_intro} is formally computed in
Section~\ref{sec:convergence}, in which we complete the proof of
Theorem~\ref{thm:main}.

\begin{table}
\begin{tabular*}{\textwidth}[t]{c|l}
    $(Z_t)_t$ & Population size process \\
    $(Z^{(u)}_t)_t$ & Size of the subfamily descending from particle $u$ \\
    $(Z^K_t)_t = (Z_t/K)_t$ & Population density process\\
    $q(z)$, $L(z)$ & Reproduction rate and offspring distribution at density $z$\\
    $m(z)$, $m_2(z)$ & First and second factorial moment of $L(z)$
\end{tabular*}
\caption{Summary of the main notation used throughout this work.}
\end{table}

\section{The moments of density-dependent branching processes}\label{sec:Topology}

\subsection{Encoding genealogies}\label{sec:marked_forests}

\subsubsection*{Planar forests}

This section introduces the topological objects and the notation used in the
remainder of the article. We encode rooted planar forests using the
Ulam--Harris formalism, \ie as subsets of $\mathcal{U}~=~
\bigcup_{n\in \N} \N^n$. For $u, v \in \mathcal{U}$, recall that $uv$
denote the concatenation of $u$ and $v$, and that $ui$ is interpreted as
the $i$-th child of $u$, for $i \geqslant 1$. The set $\mathcal{U}$ is endowed with
a partial order $\preccurlyeq$ such that, for every $u,v\in\mathcal{U}$,
$u\preccurlyeq v$ if $u$ is an ancestor of $v$:
\begin{equation*}
    u\preccurlyeq v\ \Longleftrightarrow \ \exists w\in 
    \mathcal{U} \cup \{ \varnothing \} ,\ v=uw.
\end{equation*}
We define a rooted planar forest $F$ as a subset of $\mathcal{U}$ such
that 
\begin{enumerate}[(a)]
    \item there exists a number of roots $r\geqslant 1$ such that $i\in F$ if
        and only if $i \in \{1,\ldots, r\}$;
    \item if $v\in F$, then for all $u\preccurlyeq v$ we have that $u\in F$;
    \item for every $u\in F$, there exists $d_u\in \N$ the number of offsprings of $u$ such that $ui\in F$ iff $i\leqslant d_u$.
\end{enumerate}
Finally, we denote by $\leqslant$ the lexicographical order on $\mathcal{U}$
and by $\anc{u}$ the unique parent of $u\in F \setminus \{1,\ldots, r\}$.
Even though Kingman's coalescent is defined as a non-planar tree, working
with the planar order $\leqslant$ induced by the Ulam--Harris labeling will be
convenient when deriving our recursive moment formula.

In the population model from Section~\ref{sec:model}, each particle is further
endowed with a birth time $\lambda_u$ and death time $\theta_u$ such
that $\lambda_u < \theta_u$ and $\lambda_u = \theta_{\anc{u}}$. We 
call the pair $(F, (\lambda_u, \theta_u)_{u \in F})$ a
\textit{chronological} planar forest. Recall that we denoted by 
\[
    \mathcal{N}_t\ =\ \big\{ u \in F,\ t \in [\lambda_u, \theta_u) \big\}
\]
the set of individuals alive at time $t$, and by $D_t(u, v) = t -
\theta_{u \wedge v}$ the coalescence time of $u, v \in \mathcal{N}_t$.

\subsubsection*{Planar coalescents}

Let $t\geqslant 0$, $v_1<\cdots<v_{k}$ be $k$ individuals in ${\cal N}_t$ and write ${\bm
v}\coloneqq(v_1,\ldots,v_k)$.
We encode the genealogy of the sample $\bm{v}$ by a process
$\Pi^{\bm{v}}$ with values in the partitions that we call a planar coalescent and which is constructed as follows. We refer to
Figure~\ref{fig:alternative_chrono_marked} for an illustrated example.
Formally, we let $\Pi^{\bm{v}} = (\Pi^{\bm{v}}_\tau)_{\tau \leqslant t}$ be
defined as 
\[
    i \sim_{\Pi^{\bm{v}}_\tau} j\ \iff\ D_t(v_i, v_j) < \tau
   \ \iff\ \theta_{v_i \wedge v_j} > t - \tau.
\]
Since the forest is planar and $\bm{v}$ is sorted in increasing
$\leqslant$-order, only consecutive blocks can merge at a coalescence time. This
entails that the blocks are made of consecutive integers (see
Figure~\ref{fig:alternative_chrono_marked}), and are thus also naturally
ordered. We call such an object a planar coalescent and  denote by
$\mathbb{G}^k$ the set of planar coalescents on $\{1,\ldots, k\}$.
 
Each coalescence event in a planar coalescent is entirely described by
two parameters $(i,\, d)$ such that blocks indexed from $i$ to $i+d-1$
merge. Therefore, the coalescent $\Pi^{\bm{v}}$ spanned by
the sample $\bm{v}$ can be equivalently encoded as a sequence of
coalescence times and events. 
Let $b^{\bm v}\leqslant k-1$ be the number of coalescence events occurring
between time $0$ and time $t$ in $\Pi^{\bm v}$. 
Denote by $\tau^{\bm v}_j$ the (backward) time when the $j$-th
coalescence event occurs, so that we have $0<\tau^{\bm v}_1 < \cdots <
\tau^{\bm v}_b<t$. Moreover, denote the $j$-th coalescence event as
$ c^{\bm v}_j \coloneqq (i^{\bm v}_j,\, d^{\bm v}_j)$. Again, we use the
notation $(i^{\bm v}_j,\, d^{\bm v}_j)$ to indicate that $i^{\bm v}_j,~\ldots,~i^{\bm v}_j+d^{\bm v}_j-1$ are the $d^{\bm v}_j$
consecutive lineages coalescing at time $\tau^{\bm v}_j$. 
Instead of viewing the genealogy of $\bm{v}$ as the coalescent
$\Pi^{\bm{v}}$, we equivalently encode it as the collection of random
variables
\[
\big({\bm  c}^{\bm v} ,\ \bm{\tau}^{\bm v}\big) \ \coloneqq \   \Big(\left(i_j^{\bm v},\ d_j^{\bm v}\right)_{j\leqslant b^{\bm v}},\ \big(\tau_j^{\bm v}\big)_{j\leqslant b^{\bm v}}\Big).
\]

Finally, we define $\Theta_{\tau_1}(\Pi^{\bm{v}})$ as the planar
coalescent on $\{1,\ldots,k-d^{\bm{v}}_1+1\}$ corresponding to the sequence of coalescence
times and events
\begin{equation}\label{eq:pruning:coalescent}
\Big(\big(i^{\bm v}_j,\ d^{\bm v}_j\big)_{2\leqslant j\leqslant b^{\bm v}},\ \big(\tau^{\bm v}_j-\tau^{\bm v}_1\big)_{2\leqslant j \leqslant b^{\bm v}}\Big).
\end{equation}
Intuitively, $\Theta_{\tau_1}(\Pi^{\bm{v}})$ is the made of the
lineages remaining in $\Pi^{\bm{v}}$ after the first coalescent event. We
refer once more to Figure~\ref{fig:alternative_chrono_marked} for an
illustration.

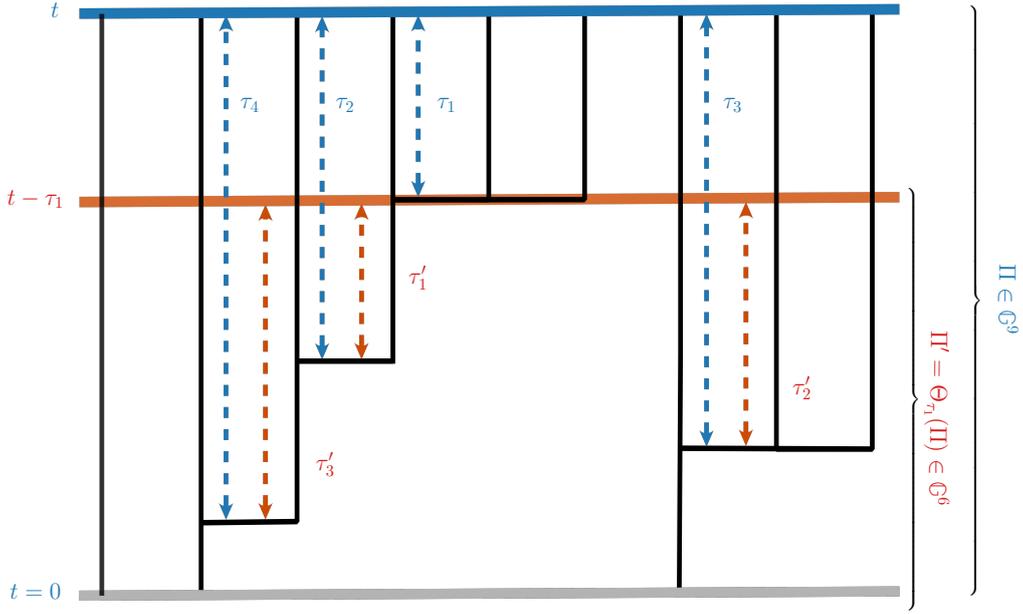
\begin{figure}
    \centering
\resizebox{0.9\textwidth}{!}{

\begin{tikzpicture}[x=0.75pt,y=0.75pt,yscale=-1,xscale=1]
\draw (396.25,257.03) node  {\includegraphics[width=501.38pt,height=372.05pt]{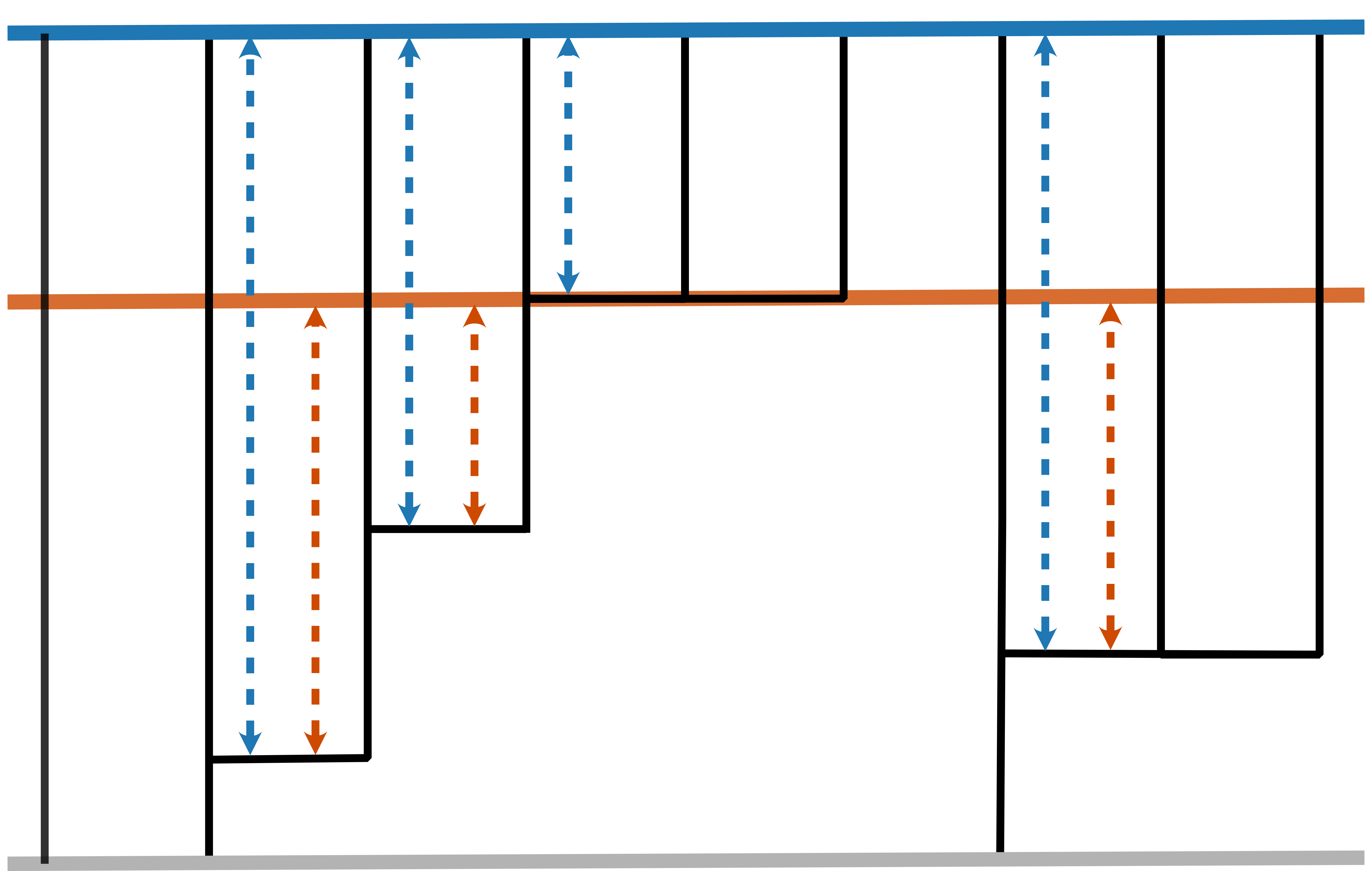}};
\draw (6,170) node [anchor=north west][inner sep=0.75pt]  [font=\Large,color={rgb, 255:red, 31; green, 119; blue, 180 }  ,opacity=1 ]  {$\textcolor[rgb]{0.84,0.11,0.11}{t-\tau _{1}}$};
\draw (41,17.4) node [anchor=north west][inner sep=0.75pt]  [font=\Large,color={rgb, 255:red, 31; green, 119; blue, 180 }  ,opacity=1 ]  {$\textcolor[rgb]{0.12,0.47,0.71}{t}$};
\draw (8,490.4) node [anchor=north west][inner sep=0.75pt]  [font=\Large,color={rgb, 255:red, 31; green, 119; blue, 180 }  ,opacity=1 ]  {${\displaystyle t=0}$};
\draw (353,95) node [anchor=north west][inner sep=0.75pt]  [font=\Large,color={rgb, 255:red, 31; green, 119; blue, 180 }  ,opacity=1 ]  {$\textcolor[rgb]{0.12,0.47,0.71}{\tau_{1}}$};
\draw (271,95) node [anchor=north west][inner sep=0.75pt]  [font=\Large,color={rgb, 255:red, 31; green, 119; blue, 180 }  ,opacity=1 ]  {$\textcolor[rgb]{0.12,0.47,0.71}{\tau_{2}}$};
\draw (584,95) node [anchor=north west][inner sep=0.75pt]  [font=\Large,color={rgb, 255:red, 31; green, 119; blue, 180 }  ,opacity=1 ]  {$\textcolor[rgb]{0.12,0.47,0.71}{\tau_{3}}$};
\draw (194,95) node [anchor=north west][inner sep=0.75pt]  [font=\Large,color={rgb, 255:red, 31; green, 119; blue, 180 }  ,opacity=1 ]  {$\textcolor[rgb]{0.12,0.47,0.71}{\tau_{4}}$};
\draw (330,231.4) node [anchor=north west][inner sep=0.75pt]  [font=\Large,color={rgb, 255:red, 31; green, 119; blue, 180 }  ,opacity=1 ]  {$\textcolor[rgb]{0.84,0.11,0.11}{\tau'_{1}}$};
\draw (255,383.4) node [anchor=north west][inner sep=0.75pt]  [font=\Large,color={rgb, 255:red, 31; green, 119; blue, 180 }  ,opacity=1 ]  {$\textcolor[rgb]{0.84,0.11,0.11}{\tau'_{3}}$};
\draw (640,321.4) node [anchor=north west][inner sep=0.75pt]  [font=\Large,color={rgb, 255:red, 31; green, 119; blue, 180 }  ,opacity=1 ]  {$\textcolor[rgb]{0.84,0.11,0.11}{\tau'_{2}}$};
\draw (730,165) node [anchor=north west][inner sep=0.75pt]  [font=\huge]  {$\begin{drcases}
 \\\\
 \\\\
 \\\\
 \\\\
\end{drcases}$};
\draw (780,16.4) node [anchor=north west][inner sep=0.75pt]  [font=\huge]  {$\begin{drcases}
 \\\\
 \\\\
 \\\\
 \\\\
 \\\\
 \\
\end{drcases}$};
\draw (770,284.11) node [anchor=north west][inner sep=0.75pt]  [font=\Large,color={rgb, 255:red, 31; green, 119; blue, 180 }  ,opacity=1 ,rotate=-90.52]  {$\textcolor[rgb]{0.84,0.11,0.11}{\Pi '=\Theta _{\tau _{1}}( \Pi) \in \mathbb{G}^{6}}$};
\draw (825,230.11) node [anchor=north west][inner sep=0.75pt]  [font=\Large,color={rgb, 255:red, 31; green, 119; blue, 180 }  ,opacity=1 ,rotate=-90.52]  {$\textcolor[rgb]{0.12,0.47,0.71}{\Pi \in \mathbb{G}^{9}}$};

\end{tikzpicture}

   }
    \caption{In blue: illustration of a planar coalescent $\Pi\in \mathbb{G}^9$ and $b=4$ branching points generated
from the sequence of mergers $\bm c=\big((4,3),(3,2),(4,3),(2,2)\big)$ and $\bm\tau = \big((\tau_j)_{1\leqslant j\leqslant 4}\big)$. The planar coalescent $\Theta_{\tau_1}(\Pi)$ obtained from pruning at the first coalescing point is pictured in red. According to \eqref{eq:pruning:coalescent}, $\Theta_{\tau_1}(\Pi)\in\mathbb{G}^{6}$ and is generated from $(\bm c', \bm\tau')=
\big(( c_2, c_3, c_4),~(\tau_2-\tau_1, \tau_3-\tau_1, 
\tau_4-\tau_1)\big)$}\label{fig:alternative_chrono_marked}
\end{figure}

\subsection{Martingale change of measure}\label{sec:spine}

The forthcoming definition introduces a new population process. It will arise as an exponential change of measure of the original process and will serve in the initialization of our recursive moment formula (see Theorem~\ref{thm:recursion_moment}).

\begin{defn}\label{def:spine}
   Fix $k\geqslant 1$ and consider the continuous-time Markov process taking
   values in $\{k, k+1, \ldots\}$ with the following transitions:
   \begin{itemize}
    \item $n\ \longmapsto\  n +  L_1(n/K) - 1$ at rate $k \cdot m(n/K)q(n/K)$, where $ L_1(n/K)$ has the size-biased offspring distribution. The latter is defined for all $z\in \R_+$ as \begin{equation}\label{eq:size_bias_distrib}
 \mathbb{P}\big(L_1(z)={n}\big)\ \coloneqq\ \frac{n}{m(z)}\mathbb{P}\big(L(z)=n\big).
\end{equation}
    \item $n\ \longmapsto\ n + L(n/K) - 1$   at rate $(n-k) \cdot q(n/K)$.
   \end{itemize}
\end{defn}
   
This process can be interpreted as the number of individuals in a
population in which $k$ distinguished particles reproduce at a modified
rate $q \cdot m$, with progeny distributed according to the size-biasing
of $L$. The remaining particles behave according to the same dynamics as
the initial population model described in Section~\ref{sec:model}. The births
from the distinguished particles can be thought of as an immigration into
the population.

Assume that we start the process with $Z_0\geqslant k$ particles and let
\[
    Z^{(i)}_t\ \coloneqq\ \#\{u\in\mathcal{N}_t,\  u\succcurlyeq i\}
\] 
be the size at time $t$ of the subfamily descending from the particle
labeled $i$ (using the Ulam--Harris labeling). 
In what
follows, for $Z \in \N$ we will always use the convention that adding a
superscript $K$ corresponds to renormalizing by $1/K$, and write for
instance 
\[
    Z^K\ \coloneqq\ Z/K.
\]

\begin{lem}[Martingale change of measure]\label{lem:martingale_change_measure_kspine}
    Let $k \geqslant 1$ and introduce the process
    \begin{equation*}
        N^{k}_t\ \coloneqq\ 
        \exp{\Big(-k\int_0^t q(Z^K_s)\big(m(Z^K_s)-1\big)ds\Big)} 
        \prod_{i=1}^k Z^{(i)}_t.
    \end{equation*}
    If neither $(Z_t)_{t \geqslant 0}$ nor the process in Definition~\ref{def:spine}
    explode, then $(N^k_t)_{t \geqslant 0}$ is a nonnegative martingale. In
    that case, if $\mathbb{P}^k$ is defined as   
    \begin{equation} \label{eq:martingaleChange}
         \frac{d\mathbb{P}^k}{d\mathbb{P}}\Big|_{\mathcal{F}_t}\ =\ N^{k}_t,
    \end{equation}
    the law of the population size $(Z_t)_{t \geqslant 0}$ under $\mathbb{P}^k$ is
    that described in Definition~\ref{def:spine}.
\end{lem}
We write $\E^k_{Z_0}[\,\cdot\,]$ for the expectation under
$\mathbb{P}^k_{Z_0}$, that is started from the initial condition $Z_0$.
We note that $N^k_t = 1$ when $k=0$ and it will be convenient to use the
convention that $\mathbb{P}=\mathbb{P}^0$, so that under $\mathbb{P}^0$
the process $(Z^K_t)_{t\geqslant 0}$ is distributed as the initial
density process (with no immigration).

\begin{proof}[Proof of Lemma~\ref{lem:martingale_change_measure_kspine}]
The case $k = 1$ has been treated by \citet[Proposition~2]{Bansaye}.
The extension to $k \geqslant 2$ follows by identical arguments. Nevertheless,
let us outline the main steps here, leaving aside some technical issues.

Denote by $\mathcal{F}_t\coloneqq \sigma\{ Z_s, s\leqslant t\}$ the canonical
filtration of $(Z_t)_{t\geqslant 0}$. First, we check that the process
$(N^{k}_t)_{t\geqslant 0}$ is indeed a nonnegative martingale for
$(\mathcal{F}_t)_{t\geqslant 0}$. The process $(N^{k}_t)_{t\geqslant 0}$ has finite
variation and applying Itô's formula it is not hard to see that for
all $i\in \{1,\ldots, k\}$, 
\[
N^{(i)}_t\ \coloneqq\ \exp{\Big(-\int_0^t
q(Z^K_s)\big(m(Z^K_s)-1\big)ds\Big)}  Z^{(i)}_t
\]
is a nonnegative
local martingale.
Note that from the description of the process, two distinct particles do
not reproduce at the same time. Hence, the processes $(N^{(i)})_{1\leqslant i\leqslant
k}$ are orthogonal, meaning that their quadratic
covariations vanish. It follows that 
  \begin{equation*} 
     N^{k}_t\ \coloneqq\ \prod_{i=1}^k N^{(i)}_t\ =\  \exp\Big(-k\int_0^t q(Z^K_s)\big(m(Z^K_s)-1\big)ds\Big)\prod_{i=1}^k Z^{(i)}_t 
\end{equation*}
is also a nonnegative local martingale.  As we assumed non-explosion of the process under both measures,  an adaptation of the arguments of \citet[Proposition 2]{Bansaye} to $k\geqslant 2$ yields that $N^{k}_t$ is a true martingale and $\mathbb{P}^k$ is well-defined as a martingale change of measure. 
  
We now identify the behavior of the population size under $\mathbb{P}^k$.  Let $h:\N^{k+1}\to \R_+$ be the function
\[
    h:\big(z, z^{(1)},\ldots, z^{(k)}\big)\ \longmapsto\ \prod_{i=1}^k z^{(i)}.
\]  
The process $(Z_t, Z^{(1)}_t, \dots, Z^{(k)}_t)$ is Markov, and if
$\mathcal{G}$ is its generator a direct computation yields
\[
  \mathcal{G}h\big(z, z^{(1)},\ldots, z^{(k)}\big)\ =\ kq(z/K)\big(m(z/K)-1\big)h\big(z, z^{(1)},\ldots, z^{(k)}\big).
\]
The process $N^k_t$ is an exponential martingale for $h$ in the sense that it can be written in the form
  \[ 
  N^k_t\ =\  \frac{h\big(Z_t, Z^{(1)}_t,\ldots, Z^{(k)}_t\big)}{h\big(Z_0, Z^{(1)}_0,\ldots, Z^{(k)}_0\big)}\exp\bigg(-\int_0^t \frac{\mathcal{G}h\big(Z_s, Z^{(1)}_s,\ldots, Z^{(k)}_s\big)}{h\big(Z_s, Z^{(1)}_s,\ldots, Z^{(k)}_s\big)} ds\bigg).
  \]
  By Itô's formula, a straightforward computation shows that
  under $\mathbb{P}^k$, the process $(Z_t, Z^{(1)}_t,\ldots, Z^{(k)}_t)_t$ is a Markov process whose infinitesimal generator is given by

\begin{equation}\label{eq:generator_changemeas_palmowski}
    \tilde{\mathcal{G}}f\ \coloneqq\ \frac{1}{h}\big[\mathcal{G}(f\cdot h)- f\cdot \mathcal{G}h\big],
\end{equation} 
for suitable functions $f:\N^{k+1} \to  \R_+$ (see  \citet[Theorem 4.2]{palmowski_technique_2002} for a similar expression).
Take a function of the first coordinate only belonging to the domain, which we denote by $f:(z,z^{(1)},\ldots, z^{(k)})\longmapsto f(z)$. Then on the one hand we have that
\begin{flalign*}
   & \frac{\mathcal{G}(f\cdot h)(z,z^{(1)},\ldots, z^{(k)})}{h(z,z^{(1)},\ldots, z^{(k)})} \\
 &\ = \ \begin{multlined}[t] \sum_{n\in\N}\sum_{i=1}^k \Big[f(z+n-1)(z^{(i)}+n-1)-f(z)z^{(i)}\Big]\frac{\prod_{j\neq i}^{k}z^{(j)}}{\prod_{i=1}^k z^{(i)}}  l_n(z/K)\cdot q(z/K)z^{(i)}\\
 +\ \sum_{n\in\N}\Big[f(z+n-1)-f(z)\Big]  l_n(z/K)\cdot q(z/K) \Big(z-\sum_{i=1}^k z^{(i)}\Big)\end{multlined}\\
    &\ =\ \sum_{n\in\N} \Big[f(z+n-1)-f(z)\Big] q(z/K)l_n(z/K) z +\sum_{n\in\N} f(z+n-1) (n-1)   l_n(z/K)  q(z/K) k,
\end{flalign*}
where we used the notation $ l_n(z/K)\coloneqq \mathbb{P}\big(L(z/K)=n\big)$.
On the other hand, 
\begin{equation*}
    \frac{f(z-k,z^{(1)},\ldots, z^{(k)})\cdot \mathcal{G}h(z-k,z^{(1)},\ldots, z^{(k)})}{h(z-k,z^{(1)},\ldots, z^{(k)})}\ =\ f(z) \sum_{n\in\N}(n-1)l_n(z/K)\cdot q(z/K) k.
\end{equation*}
Whence using \eqref{eq:generator_changemeas_palmowski} we obtain 

\begin{flalign*}
    \tilde{\mathcal{G}}f(z)\ =&\ \begin{multlined}[t]\sum_{n\in\N} \Big[f(z+n-1)-f(z)\Big] q(z/K)l_n(z/K) z  \\
   +\ \sum_{n\in\N} \Big[f(z+n-1)-f(z)\Big](n-1) l_n(z/K) \cdot q(z/K)k \end{multlined}\\
    =&\ \begin{multlined}[t]\sum_{n\in\N} \Big[f(z+n-1)-f(z)\Big] l_n(z/K) \cdot q(z/K)(z-k)\\
    +\ \sum_{n\in\N} \big[f(z+n-1)-f(z)\big] n l_n(z/K) \cdot q(z/K)k\end{multlined}\\
    =&\ \begin{multlined}[t]\sum_{n\in\N} \Big[f(z+n-1)-f(z)\Big]l_{n}\big(z/K\big)\cdot q(z/K)(z-k)\\
+\ \sum_{n\in\N} \Big[f(z+n-1)-f(z)\Big]\frac{n\cdot l_{n}(z/K)}{m(z/K)}\cdot q(z/K)m(z/K)k.\end{multlined}
\end{flalign*}
We recognize the expression of the infinitesimal generator of the Markov
process whose transitions are given in Definition~\ref{def:spine} applied to $f$,
which wraps the proof up.
\end{proof}   

\subsection{The moment formula}
Let $k\geqslant 1$, $\beta>0$, $t>0$ and $Z_0\in \N$. We define the planar
penalized $k$-moment measure $\mathbf{M}^{k,t}_{Z_0}$ as the measure on
$\mathbb{G}^k\times \R_+$ satisfying that, for $\varphi:\mathbb{G}^k\to
\R_+$ and $\psi:\R_+\to\R_+$ any two bounded and measurable functions,
\begin{equation}\label{eq:moments_planaires}
    \mathbf{M}^{k,t}_{Z_0}(\varphi, \psi)\ \coloneqq\  k!\frac{\e^{\beta Z^K_0}}{\big(Z_0\big)_k} \E_{Z_0}\bigg[ \psi\big(Z^K_t\big) \e^{-\beta Z^K_t}\sum_{\bm{v}:v_1 < \cdots < v_k \in \mathcal{N}_t} \varphi \big(\Pi^{\bm{v}}\big) \bigg],
\end{equation}
where we recall that $\Pi^{\bm{v}}\in \mathbb{G}^{k}$ stands for the planar
coalescent spanned by the ordered $k$-sample $\bm v$. We note that this
measure is finite, yet does not have unit total mass. 
When $\varphi = \1(\tau_1>t)$, that is $\varphi$ is the indicator that the
planar coalescent does not have mergers, we will simply use the notation 
$\mathbf{M}^{k,t}_{Z_0}(\psi)$ and call it the ``base case'' planar
penalized $k$-moment measure subsequently.

Since we only assume that the offspring distribution has finite second
moment, we use an exponential penalization with positive discount rate to
recover finite higher moments. We introduce the so-called penalized
factorial moments below. For $\beta>0$ and $d\in \N$, let
$m_{d}^{\beta/K}({z})$ be the $d$-th factorial moment of $L(z)$ with
exponential penalization, meaning that

\begin{equation*}
    m_{d}^{\beta/K}(z)\ \coloneqq\ \E\Big[\big(L(z)\big)_d\e^{-\beta L(z)/K}\Big].
\end{equation*}
Besides, we denote by $L^{\beta/K}_{d}(z)$ the offspring distribution biased by its penalized $d$-th factorial moment at population density $z\in\R_+$. That is, for ${n}\in \N$ we have
\begin{equation}\label{eq:factorial_moments}
 \mathbb{P}\Big(L^{\beta/K}_{d}(z)={n}\Big)\ \coloneqq\ \frac{(n)_d\cdot\e^{-\beta n/K}}{m_{d}^{\beta/K}(z)}\mathbb{P}\big(L(z)=n\big).
\end{equation}
We point out that the distribution $L^{\beta/K}_d(z)$ takes values in $\{d, d+1, \ldots\}$. For $d=1$ and $\beta=0$, we recover the usual size-biased offspring distribution $L_1(z)$ from \eqref{eq:size_bias_distrib}.

 In this article, we will say that functionals $\varphi$ on $\mathbb{G}^k, k\geqslant 2$ are \textit{of the product form} if they satisfy \eqref{eq:product_form_functionals} for some $d\in\{2,\ldots, k\}$ and $i\leqslant k-d+1$:
    \begin{equation}\label{eq:product_form_functionals}
\forall \Pi\in\mathbb{G}^{k}, \ \ 
\varphi(\Pi)\ \coloneqq\ \1\big( \tau_1<t,\,  c_1=(i,d)\big) \, \varphi_1(\tau_1) \, \varphi_2\big(\Theta_{\tau_1}(\Pi)\big),
    \end{equation}
where we recall from Section~\ref{sec:marked_forests} that $(\tau_1, c_1)$ denotes the first coalescence time and event of the planar coalescent $\Pi$. Thus, the indicator of $\{\tau_1 < t,\, c_1 = (i,d)\}$ entails that $\varphi$ charges elements of $\mathbb{G}^k$ which have at least one coalescing event in $(0,t)$ and such that the first blocks to merge are those labeled $\{i,\ldots , i+d-1\}$.

\begin{thm}[Induction for the planar moment measures]\label{thm:recursion_moment}
    Let $k \geqslant 1$ and suppose that $(Z_t)_{t \geqslant 0}$ does not explode
    under $\mathbb{P}$ nor under $\mathbb{P}^k$ defined in
    \eqref{eq:martingaleChange}. Let $(i,d)\in\N^2$ be such that
    $i+d-1\leqslant k$, and $\varphi_1,\varphi_2$ be two bounded measurable
    functions. Consider a functional $\varphi$ of the product form
    introduced in \eqref{eq:product_form_functionals}. Then, taking
    $\psi:\R_+\to \R_+$ bounded and measurable, the measure
    $\mathbf{M}^{k,t}_{Z_0}$ satisfies the following recursive formula:
    \begin{align}\label{eq:base_case}
       &\mathbf{M}^{k,t}_{Z_0}(\psi)\ =\ \e^{\beta Z^K_0}\E^k_{Z_0}\bigg[ \psi(Z^K_t)\e^{-\beta Z^K_t}\exp\Big(k\int_0^t q(Z^{K}_s) \big(m(Z^K_s)-1\big)ds\Big)\bigg] \\ \label{eq:induction}
       & \mathbf{M}^{k,t}_{Z_0}(\varphi,\psi)\  =\  \frac{1
       }{k-d+1}\binom{k}{d}\int_0^t \frac{\varphi_1(t-s)}{(Z_0-k+d-1
)_{d-1}}\mathbf{M}^{k-d+1,s}_{Z_0}\Big(\varphi_2,\,
   \mathcal{M}^{(k,d),t-s}_{\cdot}(\psi)\Big)ds.
      \end{align}
    The function of the density $\mathcal{M}^{(k,d),t-s}_{\cdot}(\psi)$
    on the right-hand side is defined as
\begin{equation}\label{eq:psi_tilde_recursive_formula}
\mathcal{M}^{(k,d),t-s}_{\cdot}(\psi):  z\ \longmapsto\ \mathcal{M}^{(k,d),t-s}_{z}(\psi)\ \coloneqq\ q(z) m^{\beta/K}_d(z)\e^{\beta/K}\E\Big[\mathbf{M}^{k,t-s}_{zK+ L^{\beta/K}_d(z)-1}(\psi)\Big],
\end{equation}
where the biased distribution $ L^{\beta/K}_d(z)$ is defined in
\eqref{eq:factorial_moments}.
\end{thm}

\begin{rem}[Many-to-few]
The recursion from Theorem~\ref{thm:recursion_moment} can be viewed as a many-to-few formula. 
One could expand this expression and obtain a very similar result to \citet{schertzer2023spectral, FST_Semipushed_2024}, in which the moment measure is directly computed from the law of a single uniform forest. 
Since our limiting object is a coalescent process which is constructed backward in time, we chose to use a backward in time recursion, conversely to the forward in time one found in \citet{schertzer2023spectral, FST_Semipushed_2024}.
\end{rem}

\begin{rem}
    As will become apparent in the proof, this recursion is
    obtained by applying the Markov property at the last branch point in
    the genealogy. The same argument could be adapted to a wide range of
    population models. As an example, for the Moran model in which the
    population size is constant, it would directly show that  its
    genealogy is Kingman's coalescent.
\end{rem}

\begin{proof}[Proof of Theorem~\ref{thm:recursion_moment}]
We proceed by induction on the number of branching points in the planar
coalescent on $\{1,\ldots, k\}$ spanned by the sampled individuals at time $t$.

We start by
determining the weight under $\mathbf{M}^{k,t}_{Z_0}$ of elements of $\mathbb{G}^k$ which do not have mergers in $(0,t)$ (that is, such that
$\tau_1>t$) which we will call the ``base case'' in the following. By
definition, the planar coalescent spanned by $\bm v$ has no mergers if
and only if each individual in the sample descends from a distinct
ancestor at time $t=0$. Thus,
\begin{equation*}
    \sum_{\bm{v}:v_1<\cdots < v_k \in \mathcal{N}_t}\1\big( \tau_1^{\bm v}>t\big)\ =\ \sum_{\bm{u}:u_1<\cdots<u_k\in\mathcal{N}_0}\, \prod_{i=1}^k Z^{(u_i)}_t ,
\end{equation*}
where $Z^{(u_i)}_t=\#\big\{ u\in \mathcal{N}_t,\,u\succcurlyeq u_i
\big\}$ is the size at time $t$ of the subfamily descending from
$u_i\in\mathcal{N}_0$. It follows from this observation and the
exchangeability of the initial particles that
\begin{equation*}
    \begin{aligned}
\mathbf{M}^{k,t}_{Z_0}(\psi)\ &=\ k!\frac{\e^{\beta Z^K_0}}{(Z_0)_k}\E_{Z_0}\bigg[ \psi(Z^K_t) \e^{-\beta Z^K_t}\sum_{\bm{v}:v_1<\cdots < v_k \in \mathcal{N}_t}\1( \tau_1^{\bm v}>t)\bigg]\\
&=\ \e^{\beta Z^K_0}\E_{Z_0}\bigg[\psi(Z^K_t)\e^{-\beta Z^K_t}\prod_{i=1}^k Z^{(i)}_t\bigg],
    \end{aligned}
\end{equation*}
where again $(Z^{(i)}_t)_i$ denote the respective sizes at time $t$ of subfamilies descending from the ancestral particles labeled by $\{1,\ldots, k\}$. 
By Lemma~\ref{lem:martingale_change_measure_kspine} we have that
\begin{align*}
    \E_{Z_0}\bigg[\psi(Z^K_t)\e^{-\beta Z^K_t}\prod_{i=1}^k Z^{(i)}_t \bigg]\ &=\ \E_{Z_0}\bigg[\psi(Z^K_t)\e^{-\beta Z^K_t}\exp\Big(k\int_0^t q(Z^K_s)\big(m(Z^K_s)-1\big)ds\Big) N^k_t \bigg]\\
   &=\ \E^k_{Z_0}\bigg[\psi(Z^K_t)\e^{-\beta Z^K_t}\exp\Big(k\int_0^t q(Z^K_s)\big(m(Z^K_s)-1\big)ds\Big)\bigg],
\end{align*}
which gives us the base case of the induction.
\medskip

Assume now that \eqref{eq:induction} holds for all forests with $b'<b$ branching points and $k'\leqslant k$ leaves. Let us consider a functional $\varphi$ of the product form given in \eqref{eq:product_form_functionals}. 

The strategy is to decompose the planar coalescent at the time $\tau^{\bm v}_1$ of its first merger. Then, we apply the induction hypothesis twice. On the one hand, to the planar coalescent $\Theta_{\tau^{\bm v}_1}(\Pi^{\bm v})$ corresponding to the genealogy of the sample starting from the first coalescence time $\tau^{\bm v}_1$, which has one fewer merger. On the other hand, to the ``trivial'' planar coalescent encoding the genealogy between backward times $0$ and $\tau^{\bm v}_1$, which corresponds to the base case of the induction.\\

The product form of $\varphi$ implies that this functional vanishes outside planar coalescents $\Pi^{\bm v}$ such that at the first coalescent event, particles labeled $v_i,\ldots, v_{i+d-1}$ find their common ancestor.

Consequently, summing over all $k$-samples $\bm v$ picked at time $t$ and such that $ c^{\bm v}_1=(i, d)$ amounts to the following: as a first step, select a sample $\bm{u}$ of size $k - d + 1$ at a time $s< t$, which satisfies that the particle in the $i$-th position undergoes reproduction before time $t$. Then, choose $d$ particles among the offspring of this $i$-th particle. Finally, single out one particle at the present time from each of the subfamilies descending from these $d$ particles, as well as from those descending from the remaining $k - d$ particles $u_j$, for $j \neq i$. 

That is,

\begin{multline}\label{eq:combinatorial_identity}
\sum_{\bm{v}:v_1<\cdots< v_k \in  \mathcal{N}_t}\1\big(\tau_1^{\bm v}<t,  c^{\bm v}_1=(i,d)\big)\varphi_1(\tau^{\bm v}_1)\varphi_2\big(\Theta_{\tau_1}(\Pi^{\bm v})\big)\ =
 \sum_{u \in \mathcal{U}, \theta_{u}<t}\varphi_1(t-\theta_{u})\\\times \sum_{\substack{\bm u: u_1< \cdots < u_{k-d+1} \in \mathcal{N}_{\theta_{u}^-}\\u_i=u}}\varphi_2(\Pi^{\bm u})\sum_{w_1<\cdots<w_d \leqslant \kappa_u} \prod_{j\leqslant d}Z^{(u w_j)}_t  \prod_{j\in [k-d+1]\setminus \{i\}} Z^{(u_j)}_t .
 \end{multline}
From there, we have

\begin{align*} 
    \mathbf{M}^{k,t}_{Z_0}(\varphi,\psi)\  &=\ k!\frac{\e^{\beta Z^K_0}}{\big(Z_0\big)_k} \E_{Z_0}\bigg[\psi(Z^K_t)\e^{-\beta Z^K_T}\sum_{\bm{v}:v_1<\cdots< v_k \in  \mathcal{N}_t}\1\big(\tau_1^{\bm{v}}<t,  c_1^{\bm v}=(i,d)\big)\varphi_1(\tau^{\bm v}_1)\varphi_2\big(\Theta_{\tau_1}(\Pi^{\bm v})\big) \bigg]
    \\ &\overset{(A)}{=}\ \begin{multlined}[t] k!\frac{\e^{\beta Z^K_0}}{\big(Z_0\big)_k}\E_{Z_0}\bigg[ \sum_{u \in \mathcal{U}, \theta_{u}<t}\varphi_1(t-\theta_{u})\sum_{\substack{\bm u:u_1< \cdots < u_{k-d+1} \in \mathcal{N}_{\theta_{u}^-}\\ u_i=u}}\varphi_2(\Pi^{\bm u})\\
    \times \sum_{w_1<\cdots<w_d \leqslant \kappa_u} \E\Big[\psi(Z^K_t)\e^{-\beta Z^K_T} \prod_{p\leqslant d}Z^{(u w_p)}_t\ \cdot  \prod_{j\neq i} Z^{(u_j)}_t   \, \big\vert\, Z_{\theta_{u}}\Big] \bigg]\end{multlined}\\
     &\overset{(B)}{=}\ \begin{multlined}[t] k!\frac{\e^{\beta Z^K_0}}{\big(Z_0\big)_k}\E_{Z_0}\bigg[\sum_{u \in \mathcal{U}, \theta_{u}<t}\varphi_1(t-\theta_{u})\sum_{\substack{\bm u: u_1< \cdots < u_{k-d+1} \in \mathcal{N}_{\theta_{u}^-}\\ u_i=u}}\varphi_2(\Pi^{\bm u}) \\
   \times\sum_{w_1< \cdots < w_d \leqslant \kappa_u}\e^{-\beta Z^K_{\theta_u}}\mathbf{M}_{Z_{\theta_{u}}}^{k,t-\theta_{u}}(\psi)\bigg]\end{multlined}\\
    &\overset{(C)}{=}\ \begin{multlined}[t] k!\frac{\e^{\beta Z^K_0}}{\big(Z_0\big)_k}\E_{Z_0}\bigg[\sum_{u \in \mathcal{U}, \theta_{u}<t}\varphi_1(t-\theta_{u})\sum_{\substack{\bm u: u_1< \cdots < u_{k-d+1} \in \mathcal{N}_{\theta_{u}^-}\\ u_i=u}}\varphi_2(\Pi^{\bm u})\e^{-\beta Z^K_{\theta_u^-}}\\
    \times\ \frac{(\kappa_u)_d}{d!}\frac{\e^{- (\kappa_u-1)\beta/K}}{ m^{\beta/K}_d\big(Z^K_{\theta_{u}^-}\big)} m^{\beta/K}_d\big(Z^K_{\theta_{u}^-}\big)\mathbf{M}_{Z_{\theta_{u}^-}+\kappa_u-1}^{k,t-\theta_{u}}(\psi)\bigg]\end{multlined}\\
    &\overset{(D)}{=}\ \begin{multlined}[t]\frac{k!}{d!}\frac{\e^{\beta Z^K_0}}{\big(Z_0\big)_k}\E_{Z_0}\bigg[\sum_{u \in \mathcal{U}, \theta_{u}<t}\varphi_1(t-\theta_{u})\sum_{\substack{\bm u: u_1< \cdots < u_{k-d+1} \in \mathcal{N}_{\theta_{u}^-}\\ u_i=u}}\varphi_2(\Pi^{\bm u})\e^{-\beta Z^K_{\theta_u^-}}\\
    \times\  m^{\beta/K}_d(Z^K_{\theta_{u}^-})\e^{\beta/K}\E\Big[\mathbf{M}_{Z_{\theta_{u}^-}+ L^{\beta/K}_d(Z_{\theta_{u}^-}/K)-1}^{k,t-\theta_{u}}(\psi)\, \big\vert\,Z_{\theta_u^-}\Big]
    \bigg].\end{multlined}
\end{align*}
In the above, step $(A)$ follows from the combinatorial identity \eqref{eq:combinatorial_identity} and conditioning on the population size at time $\theta_{u}$, $(B)$ from an application of the strong Markov property at time $\theta_{u}$ and recognizing the base case of the induction over $(\theta_u, t)$. To obtain line $(C)$, we used that $Z_{\theta_u}= Z_{\theta_u^-}+\kappa_u-1$, we planarized the progeny and introduced the penalized factorial moment of order $d$. Finally, $(D)$ follows from the fact that $\kappa_u$ is independent of $\theta_u$ conditionally upon $Z_{\theta_u^-}$. We also used that $\kappa_u$ is distributed according to $L\big(Z_{\theta_u^-}/K\big)$, which we biased by $(\kappa_u)_d \e^{-\beta \kappa_u/K}$ by choosing $d$ ordered distinguished lineages among the progeny.
Each particle in $\mathcal{N}_s$ dies at the same rate $q(Z^K_s)$. Therefore, compensating for the jumps we obtain that

\begin{align*}
    \mathbf{M}^{k,t}_{Z_0}(\varphi,\psi) 
   \ & =\ \begin{multlined}[t] \frac{k!}{d!}\e^{\beta Z^K_0}\frac{\big(Z_0\big)_{k-d+1}}{\big(Z_0\big)_k}\E_{Z_0}\bigg[\int_0^t \sum_{u \in \mathcal{N}_s}q(Z^{K}_s)\e^{-\beta Z^K_s} \frac{\varphi_1(t-s)}{(Z_0)_{k-d+1}}\sum_{\substack{ u_1< \cdots < u_{k-d+1} \in \mathcal{N}_{s}\\ u_i=u}}\varphi_2(\Pi^{\bm u})\\
   \times m^{\beta/K}_d(Z^{K}_s)\e^{\beta/K}\, \E\Big[ \mathbf{M}^{k,t-s}_{Z_s-1+ L^{\beta/K}_d(Z^K_s)}(\psi)\, \big\vert \, Z_s\Big]\bigg]ds\quad\end{multlined}\\
    &=\ \frac{k!}{d!}\int_0^t  \frac{\varphi_1(t-s)\e^{\beta Z^K_0}}{(Z_0-k+d-1)_{d-1}}\E_{Z_0}\bigg[\frac{\e^{-\beta Z^K_s}}{(Z_0)_{k-d+1}} \mathcal{M}^{(k,d),t-s}_{Z^K_s}(\psi) \sum_{ u_1 < \cdots < u_{k-d+1} \in \mathcal{N}_{s}}\varphi_2(\Pi^{\bm u})\bigg]\\
    &=\ \frac{1}{k-d+1}\binom{k}{d} \int_0^t  \frac{\varphi_1(t-s)}{(Z_0-k+d-1)_{d-1}} \mathbf{M}^{k-d+1,s}_{Z_0}\big(\varphi_2, \mathcal{M}^{(k,d),t-s}_{\cdot}(\psi) \big)ds.
\end{align*}
This yields \eqref{eq:induction} for planar chronological forests with $k$ leaves and $d$ branching points, thereby concluding the proof of the proposition.
\end{proof}

\section{Stochastic analysis of the population size}\label{sec:calcul_sto}
We recall that we use the convention that $\mathbb{P}^0=\mathbb{P}$ the distribution of the original density process.

\subsection{The population density under \texorpdfstring{$\mathbb{P}^k$}{Pk}}
In this section, we fix $k\geqslant 0$. We recall that $L_1(z)$ is the size-biased offspring distribution at
density $z$ from Definition~\ref{def:spine}.

It is not hard to see from the transitions described earlier that, under $\mathbb{P}^k$, the density process $(Z^K_t)_{t\geqslant 0}=(Z_t/K)_{t\geqslant 0}$ has the same
distribution as the solution to the following stochastic differential
equation:
\begin{multline}\label{eq:SDE_under_Q}
    Z^K_{t}-Z^K_0\ =\ \int_{0}^{t} \int_{\R_+}\int_0^1 \frac{L(Z^K_{s^-},\xi)-1}{K}  \1\big(\varrho\leqslant q({Z}^K_{s^-})(Z_{s^-}-k)\big)\mathcal{Q}_1(dsd\varrho d\xi)\\
    +\ \int_{0}^{t}\int_{\R_+}\int_0^1  \frac{L_1(Z^K_{s^-},\xi)-1}{K}  \1\big(\varrho\leqslant k q(Z_{s^-}^K) m(Z^K_{s^-})\big)\mathcal{Q}_2(dsd\varrho d\xi),
\end{multline}
where $\mathcal{Q}_1$ and $\mathcal{Q}_2$ are independent Poisson processes on $\R_+\times\R_+\times (0,1)$ with intensity $ds\otimes d\varrho \otimes d\xi$ (see \cite{ikeda_watanabe}). Above, $(z,\xi)\longmapsto L(z,\xi)$ and $(z,\xi)\longmapsto L_1(z,\xi)$ are  nonnegative measurable functions chosen in such a way that, if  $U\sim \mbox{Unif}([0,1])$, the random variables $L(z,U)$ and $L_1(z,U)$ have the same distribution as $L(z)$ and $L_1(z)$, respectively.

 Introducing the compensated measures $ \bar{\mathcal{Q}}_i(dsd\varrho d\xi)= \mathcal{Q}_i(dsd\varrho d\xi)-dsd\varrho d\xi$ for $i\in\{1,2\}$ and recalling that $m(z)=\E[L(z)]$ and $m_2(z)=\E[L(z)(L(z)-1)]$, we obtain
\begin{equation*}
       Z^K_{t}-Z^K_0\ =\ \int_0^{t} \big({m}({Z}^K_s)-1\big)q({Z}^K_s){Z}^K_s ds + \int_0^{t} \frac{k}{K}  q({Z}^K_s){m}_2({Z}^K_s) ds + W^{(k)}_t,
\end{equation*}
where the process $ W^{(k)}$ is defined by

\begin{multline}\label{eq:martingale_before_cuting_in_three_parts} 
W^{(k)}_t\ =\ \int_0^t \int_{\R_+} \int_0^1 \frac{L(Z^K_{s^-},\xi)-1}{K}\1\big(\varrho \leqslant q({Z}^K_{s^-})({Z}_{s^-}-k)\big)\bar{\mathcal{Q}}_1( ds d\varrho d\xi)\\
   + \int_0^t \int_{\R_+} \int_0^1 \frac{L_1(Z^K_{s^-},\xi)-1}{K} \1\big(\varrho \leqslant k q({Z}^K_{s^-}) m({Z}^K_{s^-})\big)  \bar{\mathcal{Q}}_2( ds d\varrho d\xi).
\end{multline}

\subsection{Convergence to the equilibrium density}\label{sec:deviations}
For any $\gamma > 0$, we denote by $\mathscr{B}_{\gamma} \coloneqq \{z,\ |z-1|<\gamma\}$, and introduce the exit time
\begin{equation}\label{eq:stopping_times}
      T_{\gamma} \ \coloneqq\ \inf\big\{t\geqslant 0,\ {Z}^K_t\notin \mathscr{B}_{\gamma}\big\}.
\end{equation}

\begin{prop}[Convergence of the population density]\label{prop:coupling_frozen}
Let $(Z_0)_{K}$ be a sequence of initial conditions satisfying $Z_0/K\to 1$. Then for every $T>0$, $k\geqslant 0$ and $\gamma>0$,
\begin{equation*}
    \mathbb{P}^k_{Z_0}(T_{\gamma}<KT\big)\ \underset{K\to\infty}{\longrightarrow}\ 0.
\end{equation*}
\end{prop}

The proof of Proposition~\ref{prop:coupling_frozen} relies on Lemma~\ref{lem:mdp_smalllem}. 
In turn, the proof of the latter result uses Lemma~\ref{lem:coupling_capped} below.
\begin{rem}\label{rem:technicalities}
    We point out that Lemma~\ref{lem:coupling_capped} and
    Lemma~\ref{lem:mdp_smalllem} are technical results which we use to
    re-establish the estimates of \cite{forien2025} in a more general
    setting. Indeed, we do not make any assumption on the offspring
    distribution besides Assumption~\ref{hyp:equilibrium} and Assumption~\ref{hyp:moments}.
    In particular, the offspring distribution might be irregular at $z=1$
    as we only impose continuity of its first and second moments. At this
    level of generality, to prove convergence of the density on the
    evolutionary timescale we decompose the jumps according to their size
    and couple the original process to a population model confined in a
    macroscopic neighborhood of the equilibrium.
\end{rem}

\begin{lem}\label{lem:coupling_capped}
    Under Assumption~\ref{hyp:moments}, there exists a sequence
    $(\Gamma_K)_K$ such that
    \[
        \lim_{K \to \infty} \frac{\Gamma_K}{K}\ =\ 0,
        \qquad
        \adjustlimits{\lim}_{K \to \infty}{\sup}_{z \in (1-\varepsilon, 1+\varepsilon)} 
        K^2 \mathbb{P}( L(z) > \Gamma_K)\ =\ 0.
    \]
For this sequence we have in addition that 
\[
\adjustlimits{\lim}_{K \to \infty}{\sup}_{z \in (1-\varepsilon, 1+\varepsilon)} 
        K m(z) \mathbb{P}\big(L_1(z) > \Gamma_K\big)\ =\ 0
\]
\end{lem}

\begin{proof}
    For any $\eta > 0$,
    \[
        K^2 \mathbb{P}( L(z) \geqslant \eta K) 
      \  \leqslant \
        \frac{1}{\eta^2} \E\big[ L(z)^2 \1_{\{ L(z) \geqslant \eta K\}} \big].
    \]
    Therefore, by \eqref{eq:UI_moment},
    \[
         \adjustlimits{\lim}_{K \to \infty}{\sup}_{z \in (1-\varepsilon, 1+\varepsilon)} K^2 \mathbb{P}( L(z) \geqslant \eta K) 
      \ =\ 0.
    \]
    From that, one can extract a sequence $(\eta_K)_K$ such that
    $\eta_K \to 0$ and 
    \[
        \adjustlimits{\lim}_{K \to \infty}{\sup}_{z \in (1-\varepsilon, 1+\varepsilon)}  K^2 \mathbb{P}( L(z) \geqslant \eta_K K) 
       \ =\ 0.
    \]
    In addition, for all $z \in (1-\varepsilon, 1+\varepsilon)$,
    \[
     K m(z) \mathbb{P}( L_1(z) \geqslant \eta_K K) \ \leqslant\ \sup_{z \in (1-\varepsilon, 1+\varepsilon)}  K^2 \eta_K \mathbb{P}( L(z) \geqslant \eta_K K). 
    \]
    Thus,
    \begin{equation*}
        \adjustlimits{\lim}_{K \to \infty}{\sup}_{z \in (1-\varepsilon, 1+\varepsilon)}  K m(z) \mathbb{P}( L_1(z) \geqslant \eta_K K)\  = \ 0.
    \end{equation*}
The result follows by setting $\Gamma_K = \eta_K K$ and from the
    definition of $L_1(z)$. 
\end{proof}

With this, we can now state the subsequent concentration result. 
\begin{lem}\label{lem:mdp_smalllem} 
  For all $T>0$,  $\delta>0$ and $ \gamma>0$ and any sequence of initial population sizes $(Z_0)_K$ for which $Z_0/K\in \mathscr{B}_{(1-\delta)\gamma}$, we have

\begin{equation*}
   \lim_{K\to\infty} K\Big(1-\mathbb{P}^k_{Z_0}\big( T_{\gamma}> T , \  Z_{T}^K\in \mathscr{B}_{
    (1-\delta)\gamma}\big)\Big) \ = \ 0.
\end{equation*} 

\end{lem}

\begin{proof}
    Recall the parameter $\varepsilon$ from condition \eqref{eq:UI_moment}. Since $\gamma\mapsto T_{\gamma}$ is increasing, there is no loss of generality in taking $\gamma < \varepsilon$.
    The proof relies on  a modification of $(Z^K_t)_{t \geqslant 0}$ in which
    the large birth events are removed and which is ``frozen'' upon
    exiting a neighborhood of the equilibrium. Let
    $(\bar{Z}^K_t)_{t \geqslant 0}$ be the solution to \eqref{eq:SDE_under_Q},
    but with modified birth rate 
    \[
   \bar q: \begin{cases} 
        \bar{q}(z)\ =\ q(z),\ z\in \mathscr{B}_{\gamma}\\ 
        \bar{q}(z)\ =\ 0, \ z\notin \mathscr{B}_{\gamma},
    \end{cases}
    \]
    and modified offspring distributions
    \[
        \bar{L}(z, \xi)\ = \ \1_{\{ L(z, \xi) \leqslant \Gamma_K \}} L(z, \xi) 
        + \1_{\{ L(z, \xi) > \Gamma_K \}},
        \quad
        \bar{L}_1(z, \xi)\ =\  \1_{\{ L_1(z, \xi) \leqslant \Gamma_K \}} 
        L_1(z, \xi) + \1_{\{ L_1(z, \xi) > \Gamma_K \}}.
    \]
    We introduce
    \begin{multline*}
        T'_\Gamma \ =\ \inf\Big\{ t > 0 : 
            \int_0^t \int_{\R_+} \int_0^1 
            \1\big(L(Z_{s^-}^K, \xi) > \Gamma_K, \varrho \leqslant (Z_{s^-}-k)q(Z_{s^-}^K)\big)
            \mathcal{Q}_1(dsd\varrho d\xi) \\
            +
            \int_0^t \int_{\R_+} \int_0^1 
            \1\big(L_1(Z_{s^-}^K, \xi) > \Gamma_K, \varrho \leqslant k
            q(Z_{s^-}^K)m(Z^K_{s^-}) \big)
            \mathcal{Q}_2(dsd\varrho d\xi) > 0
        \Big\}
    \end{multline*}
    the first time when a birth/immigration event of more than $\Gamma_K$
    offspring occurs. We also let $\bar{T}_\gamma$ and $\bar{T}'_\Gamma$ be
    defined as $T_\gamma$ and $T'_\Gamma$ but using $(\bar{Z}^K_t)_t$. It
    is clear that the two processes coincide until time $T_\gamma
    \wedge T'_\Gamma$. Conditional on the trajectory of $(\bar Z^K_t)_t$, the process giving the time of occurrence of the jumps larger than $\Gamma_K$ is a Cox point process with intensity
    \[q(\bar{Z}^K_s) \Big((\bar{Z}_s-k)
                \mathbb{P}\big(L(\bar{Z}^K_s) > \Gamma_K \mid \bar{Z}_s\big)+ k  m(\bar{Z}^K_{s})
                \mathbb{P}\big(L_1(\bar{Z}^K_s) > \Gamma_K \mid \bar{Z}_s\big)\Big)ds.\] 
                It follows that
    \begin{align*}
       \mathbb{P}^k_{Z_0}( T_\gamma > T) 
       \ &\geqslant\  
       \mathbb{P}^k_{Z_0}( \bar{T}_\gamma > T,\ \bar{T}'_\Gamma > T) \\
       &=\E^k_{Z_0}\Big[\1_{\{\bar{T}_\gamma > T\}}\mathbb{P}\big(\bar{T}'_\Gamma > T \, \vert\, (\bar{Z}_s)_{s\leqslant T} \big)\Big]  \\
        &=\ \begin{multlined}[t]
         \E^k_{Z_0}\bigg[ \1_{\{\bar{T}_\gamma > T\}} 
            \exp\Big( -\int_0^T q(\bar{Z}^K_s) (\bar{Z}_s-k)
                \mathbb{P}\big(L(\bar{Z}^K_s) > \Gamma_K \mid \bar{Z}_s\big)\\
                + k q(\bar{Z}^K_s) m(\bar{Z}^K_{s})
                \mathbb{P}\big(L_1(\bar{Z}^K_s) > \Gamma_K \mid \bar{Z}_s\big)
        ds \Big) \bigg].
        \end{multlined}
    \end{align*}
As we picked $\gamma\in(0, \varepsilon)$, on the event $\{\bar{T}_\gamma > T\}$ we have $|\bar{Z}^K_s-1|< \varepsilon$. Therefore,
    \begin{align*}
       \mathbb{P}^k_{Z_0}( T_\gamma > T)   \  &\geqslant\ \begin{multlined}[t]
        \mathbb{P}^k_{Z_0}( \bar{T}_\gamma > T)\bigg[1 - T\Big( 
            (1+\varepsilon) K \sup_{z \in (1-\varepsilon, 1+\varepsilon)} 
            q(z)\mathbb{P}\big(L(z) > \Gamma_K\big) \\
            + k \sup_{z \in (1-\varepsilon, 1+\varepsilon)} 
        q(z)m(z)\mathbb{P}\big(L_1(z) > \Gamma_K\big)\Big)\bigg].
        \end{multlined}
    \end{align*}
    By choosing $(\Gamma_K)_K$ as in Lemma~\ref{lem:coupling_capped}, the
    result is proved if we can show that 
    \[
        \lim_{K\to\infty} K \mathbb{P}^k_{Z_0}( \bar{T}_\gamma > T)\ =\ 0.
    \]
    It will be convenient to introduce the notation
    \[
        \bar{m}(z)\ =\ \E\big[\bar{L}(z)\big],\qquad \bar{m}_2(z)\ =\
        \E\big[\bar{L}(z)(\bar{L}(z)-1)\big].
    \]
The process $(\bar Z^K_t)_t$ solves an SDE obtained by replacing $q,\ L$ and $L_1$ by $\bar{q},\ \bar L$ and $\bar L_1$ in \eqref{eq:SDE_under_Q}.
By linearizing it around the equilibrium density, we have
\begin{align}\nonumber
    d\bar Z^K_t\ 
    &= \ 
    (\bar{m}(\bar{Z}^K_t)-1) \bar{q}(\bar{Z}^K_t) \bar{Z}^K_t dt
    + \frac{k}{K} \bar{q}(\bar{Z}^K_t) \bar{m}_2(\bar{Z}^K_t) dt 
    + d\bar W^{(k)}_t \\ \label{eq:linearized_sde}
    &=\  
    \alpha(\bar{Z}^K_t) (\bar{Z}^K_t - 1) dt 
    + \Big(\frac{k}{K} \bar{m}_2(\bar Z^K_t)+ \bar{m}(\bar{Z}^K_t)-m(\bar{Z}^K_t)\Big)\bar q(\bar{Z}^K_t) \bar{Z}^K_t dt 
    + d\bar W^{(k)}_t,
\end{align}
where for $z\in \R_+$,  $\alpha(z) = \frac{m(z)-1}{z-1}\bar{q}(z)z$.
Under Assumption~\ref{hyp:equilibrium}, the mean value theorem gives the existence
of a constant $\bar \alpha > 0$ such that  
\[
    \sup_{z\in \mathscr{B}_{\gamma}} \alpha(z) \leqslant -\bar \alpha.
\]
With this, we handle the various terms in 
\eqref{eq:linearized_sde} separately. 

First, note that $\bar{m}_2(z) \leqslant m_2(z)$ and
\[
    0 \ \leqslant \ m(z) - \bar{m}(z)\ \leqslant\ \E\big[L(z)\1_{\{L(z) > \Gamma_K\}}\big]\ \leqslant\ 
    \E\big[L(z)^2\big]^{1/2} \mathbb{P}\big(L(z) > \Gamma_K\big)^{1/2}.
\]
Therefore, by Lemma~\ref{lem:coupling_capped} and the fact that $\bar{q}(z)=0$ for $z\notin \mathscr{B}_{\gamma}$, there exists $C_1$ such that
almost surely,
\[
   \Big| \frac{k}{K} \bar{m}_2(\bar Z^K_t) +
    \bar{m}(\bar{Z}^K_t) - m(\bar{Z}^K_t) \Big| \bar{q}(\bar{Z}_t^K) \bar{Z}_t^K 
   \ \leqslant\  \frac{C_1}{K},
\]

Second, we control the martingale terms using a version of the
Burkhölder--Davis--Gundy inequality from
\citet[Theorem 2.1]{ma_elena_hernandez-hernandez_generalisation_2022}.
This requires us to control the moments of its predictable quadratic
variation $\big(\big\langle \bar{W}^{(k)}, \bar{W}^{(k)} \big\rangle _t\big)_{t \geqslant 0}$ as well
as the compensator of $\big(\sum_{s \leqslant t} \big|\Delta
\bar{W}^{(k)}_s\big|^{2p}\big)_{t \geqslant 0}$, for some $p > 1$. For $p=2$, the
latter compensator is
\begin{align*}
    A_t\ 
    &=\ \begin{multlined}[t]
    \frac{1}{K^4} \int_0^t 
    \E\Big[\big(\bar{L}(\bar Z^K_s)-1\big)^{4}\, \vert\,
\bar{Z}_s\Big] \bar{q}(\bar Z^K_s) (\bar Z_s-k) ds \\
    + \frac{1}{K^4} \int_0^t 
    \E\Big[\big(\bar{L}^{(1)}(\bar Z^K_s)-1\big)^{4}\, \vert\,
\bar{Z}_s\Big] k \bar{m}(\bar Z^K_s) \bar{q}(\bar Z^K_s) ds
    \end{multlined} \\
    &\leqslant\ \begin{multlined}[t]
    \Big( \frac{\Gamma_K}{K} \Big)^2 \frac{1}{K} \int_0^t 
    \E\big[L(\bar Z^K_s)^{2}\, \vert\,
\bar{Z}_s\big] \bar{q}(\bar Z^K_s) \big(\bar Z^K_s - \tfrac{k}{K}\big) ds \\
    +  \Big( 
         \frac{\Gamma_K}{K} \Big)^3 \frac{1}{K} \int_0^t  
    \E\big[L(\bar Z^K_s)^{2}\, \vert\,
\bar{Z}_s\big] k \bar{m}(\bar Z^K_s) \bar{q}(\bar Z^K_s) ds
    \end{multlined} \\
    &\leqslant\  \Big( \frac{\Gamma_K}{K} \Big)^2 \frac{C_2}{K},
\end{align*}
for some $C_2 > 0$. Similarly, 
\begin{align*}
   \big\langle \bar{W}^{(k)}, \bar{W}^{(k)} \big\rangle _t\ 
   &=\ \begin{multlined}[t]
   \frac{1}{K} \int_0^t 
   \E\Big[\big(\bar L(\bar Z^K_s)-1\big)^2\, \vert\,\bar{Z}^K_s \Big] 
   \bar{q}(\bar Z_s)\big(\bar Z^K_s - \tfrac{k}{K} \big)ds \\
   + \frac{1}{K^2} \int_0^t  \E\Big[\big(\bar{L}_1(\bar{Z}^K_s)-1\big)^2\, \vert\,\bar{Z}_s\Big] \bar{q}(\bar Z^K_s) k ds
   \end{multlined} \\
   &\leqslant\ \frac{C_3}{K},
\end{align*}
for some $C_3 > 0$.
Using Markov's inequality combined with \citet[Theorem~2.1]{ma_elena_hernandez-hernandez_generalisation_2022}, for any $c >0$, there exist $C_4$, $C_5>0$ such that, 
\begin{equation*}
\begin{aligned}
    \mathbb{P}^k_{Z_0}\Big(\sup_{t\leqslant T}\ |\bar{W}^{(k)}_t|\geqslant c \Big)\ &\leqslant\ C_4 \E^k_{Z_0}\Big[\big\langle \bar{W}^{(k)}, \bar{W}^{(k)} \big\rangle _T^{2}+ A_T\Big] \\
    &\leqslant\ \frac{C_5}{K} \Big(\frac{\Gamma_K}{K}\Big)^2,
\end{aligned}
\end{equation*}
which is negligible compared to $1/K$ as $K\to \infty$ by our choice of $(\Gamma_K)_K$.
Finally, going back to \eqref{eq:linearized_sde} and working on the event that
$\sup_{t\leqslant T}\ |\bar{W}^{(k)}_t| \leqslant c$, the first two steps above show that for all $t\leqslant T$,
\begin{equation}\label{eq:before_gronwall}
   d( \bar{Z}^K_t - 1) \ \leqslant\ 
   2 c dt
     -  \bar\alpha\cdot  (\bar Z^K_t-1) dt.
\end{equation}
By an application of Gr\"onwall's inequality, we have that the final value $\bar Z^K_T$ satisfies
\begin{align*}
 |\bar{Z}^K_T - 1 | \ &\leqslant\ 
   \e^{-\bar \alpha T}|\bar Z^K_0-1|+  \frac{2c}{\bar \alpha}(1-\e^{-\bar \alpha T})\\
   &<\  \e^{-\bar \alpha T}(1-\delta)\gamma+\frac{2 c}{\bar \alpha}\\
   &<\ (1-\delta)\gamma,
\end{align*}
for $c$ small enough.
The same equation yields for all $t\geqslant 0$ that 
\[
 |\bar{Z}^K_t - 1 | \ \leqslant\ 
   |\bar Z^K_0-1|
   +  \frac{2c}{\bar \alpha}\ 
   <\ \gamma,
\]
which concludes the proof of the lemma.
\end{proof}

The proof of Proposition~\ref{prop:coupling_frozen} now consists in iterating Lemma~\ref{lem:mdp_smalllem} over $K$ intervals of time with fixed length $T$, with the help of the Markov property. We point out that this proof uses the same idea as in a recent work of \cite{Julie_Zsofia}.

\begin{proof}[Proof of Proposition~\ref{prop:coupling_frozen}]
For $i\leqslant K$ we denote by $T_i\coloneqq iT$ and consider the sequence of events
\begin{equation*}
    \mathcal{A}_i\ \coloneqq\ \big\{ T_{\gamma} > T_i,\ Z^K_{T_i} \in \mathscr{B}_{(1-\delta)\gamma}\big\}.
\end{equation*}
In addition, for all $i\in \{2,\ldots,K \}$ we use the Markov property at time $T_i$ to derive that
\begin{align*}
    \mathbb{P}^k_{Z_0}\big(\mathcal{A}_{i+1}^c\big)\ &=\ \mathbb{P}^k_{Z_0}\big(\mathcal{A}_{i+1}^c\cap \mathcal{A}_{i}\big)+\mathbb{P}^k_{Z_0}\big(\mathcal{A}_{i}^c\big)\\
    &\leqslant\ \mathbb{P}^k_{Z_{0}}\big(\mathcal{A}_{i} \cap \{T_{\gamma}<T_{i+1} \cup {Z}^K_{T_{i+1}}\not\in \mathscr{B}_{(1-\delta)\gamma} \}\big)  +  \mathbb{P}^k_{Z_{0}}\big(\mathcal{A}_{i}^c\big) \\
    &\leqslant\ \mathbb{P}^k_{Z_{0}}\big(\mathcal{A}_{i} \cap \{T_{\gamma}\in (T_i, T_{i+1}) \cup {Z}^K_{T_{i+1}}\not\in \mathscr{B}_{(1-\delta)\gamma} \}\big)  +  \mathbb{P}^k_{Z_{0}}\big(\mathcal{A}_{i}^c\big) \\
    &\leqslant\ \sup_{z \in \mathscr{B}_{(1-\delta)\gamma}} \mathbb{P}^k_{zK}\big(\mathcal{A}_{1}^c\big) +  \mathbb{P}^k_{Z_{0}}\big(\mathcal{A}_{i}^c\big) .
\end{align*}
Hence by a short induction, for all $j$,
$\mathbb{P}^k_{Z_{0}}\big(\mathcal{A}^c_j\big)\leqslant  j\sup_{z \in \mathscr{B}_{(1-\delta)\gamma}}\mathbb{P}^k_{zK}\big(\mathcal{A}^c_1\big)$ and 
\begin{align*}
    \mathbb{P}^k_{Z_{0}}\big( T_{\gamma} < KT \big)
    \ &\leqslant\
    \mathbb{P}^k_{Z_{0}}\big( \mathcal{A}^c_K \big) 
    \\
    &\leqslant\ K\sup_{z \in \mathscr{B}_{(1-\delta)\gamma}}\big(1- \mathbb{P}^k_{zK}\big(\mathcal{A}_1\big)\big),
\end{align*}
which by Lemma~\ref{lem:mdp_smalllem} vanishes as $K\to\infty$, for all initial population sizes $Z_0$ satisfying $Z_0/K\in \mathscr{B}_{(1-\delta)\gamma}$, hence taking $K$ large enough, for all initial condition satisfying $Z_0/K\to 1$.
\end{proof}

\subsection{Convergence to Brownian motion}\label{sec:densityprocess}
The main ingredient coming into play to identify Kingman's coalescent is
to estimate the exponential term in the base case of the moment formula
in Theorem~\ref{thm:recursion_moment} by studying the population density
process. Here, we give some estimates on the population density process
$(Z^K_t)_t$. Recall from \eqref{eq:SDE_under_Q} the semimartingale
decomposition of the population density under the measure $\mathbb{P}^k_{Z_0}$, for sequences of initial sizes $(Z_0)_K$ satisfying $Z_0/K\to 1$.

\begin{lem}\label{lem:cv_martingale_jacod_shiryaev}
Fix $k\geqslant 0$ and recall from \eqref{eq:martingale_before_cuting_in_three_parts} the martingale part $(W^{(k)}_t)_{t\geqslant 0}$ of the density process under $\mathbb{P}^k_{Z_0}$. 
   Then under $\mathbb{P}^k_{Z_0}$, 
   \[
   \big(W^{(k)}_{Kt}\big)_{t\geqslant 0}\ \underset{K\to\infty}{\longrightarrow}\ \big(\sqrt{q(1)m_2(1)}B_t\big)_{t\geqslant 0},
   \]
   in distribution for the topology of uniform convergence on compact sets and where $B$ is a standard Brownian motion. Moreover, the quadratic variation process 
   \[
   \big([W^{(k)},W^{(k)}]_{Kt}\big)_{t\geqslant 0}\ \coloneqq\ \Big(\sum_{s\leqslant Kt}|\Delta W^{(k)}_s|^2\Big)_{t\geqslant 0}
   \] 
   converges in probability to $\big( q(1)m_2(1)t\big)_{t\geqslant 0}$ as $K\to\infty$. 
\end{lem}

\begin{proof}
The result will follow from the martingale central limit theorem.

By Proposition~\ref{prop:coupling_frozen}, up to freezing the process when it exits a macroscopic ball $\mathscr{B}_{\gamma}$ for some $\gamma > 0$, we can assume that the reproduction rate $q$ vanishes outside $\mathscr{B}_{\gamma}$. Another consequence of Proposition~\ref{prop:coupling_frozen} is that  
\[
\sup_{s < Kt} |\Delta W^{(k)}_s| \ \to \ 0,
\]
 in probability, as otherwise the large jumps would make the density escape $\mathscr{B}_{\gamma}$.  Therefore, we can extract a sequence $(\eta_K)_K$ such that $\eta_K\to 0$ and 

 \[
\mathbb{P}\big(\sup_{s < Kt} |\Delta W^{(k)}_s| >  \eta_K\big)\ \to \ 0.
\]

From there, we decompose the 
process $W^{(k)}$ into three
parts, namely
\[
   W^{(k)}_t\ =\ W'_t+W''_t+ W'''_t, 
\]
where
\begin{align*}
   & W'_t\ =\ \int_0^t \int_{\R_+} \int_0^1 \frac{L(Z^K_{s^-},\xi)-1}{K}\1\big(\varrho \leqslant q({Z}^K_{s^-})({Z}_{s^-}-k)\big)\bar{\mathcal{Q}}_1( ds d\varrho d\xi)\\
   & W''_t\ =\ \int_0^t \int_{\R_+} \int_0^1 \frac{L_1(Z^K_{s^-},\xi)-1}{K}\1\big( L_1(Z^K_{s^-},\xi) \leqslant \eta_K K\big)\1\big(\varrho \leqslant k q({Z}^K_{s^-}) m({Z}^K_{s^-})\big)  \bar{\mathcal{Q}}_2( ds d\varrho d\xi) \\
   & W'''_t\ =\ W^{(k)}_t-W'_t-W''_t.
\end{align*}

The martingales $ W'$ and $ W''$ are square-integrable. Re-scaling time by $K$, the predictable quadratic variation (or angle bracket) of $W'$ satisfies the following: 
      \begin{equation}\label{eq:angle_bracket}
        \big\langle W', W' \big\rangle _{Kt}
        \  =\ \int_0^{t} \big(Z^K_{Ks}- k/K)q(Z^K_{Ks}\big)\big(m_2(Z^K_{Ks})-m({Z}^K_{Ks})+1\big)ds.
      \end{equation}
Since $Z_{Ks}^K\to 1$ uniformly on $(0,t)$ in probability under $\mathbb{P}^k_{Z_0}$,
using that $q,\, m$ and $m_2$ are  continuous at $1$, the integrand in
\eqref{eq:angle_bracket} converges uniformly to $q(1)m_2(1)$ for $s$ in
$(0,t)$. Thus, the angle bracket process
$\big( \big\langle W', W' \big\rangle _{Kt}\big)_t$ converges to $\big(q(1)m_2(1)t\big)_t$
almost-surely and uniformly on compact intervals.

Besides, the predictable quadratic variation of $W''$ can be
upper-bounded as below:
\begin{align*}
 \big\langle W'', W'' \big\rangle _{tK}\ &=\ \int_0^{tK} \E\Big[(L_1(Z^K_s)-1)^2\1(L_1(Z^K_s)\leqslant \eta_K K)\,\vert\, Z_s \Big] \frac{k}{K^2} q(Z^K_s)m(Z^K_s) ds\\
  & \leqslant\ \int_0^{tK}  \E\Big[(L_1(Z^K_s)-1)\1(L_1(Z^K_s)\leqslant \eta_K K)\,\vert\, Z_s \Big]  \frac{\eta_K}{K} k q(Z^K_s)m(Z^K_s)  ds\\
  &\leqslant\ \eta_K \int_0^{t} k q(Z^K_{Ks})m_2(Z^K_{Ks}) ds,
\end{align*}
which converges in probability to zero, since $\eta_K=o(1)$.

Since the jump sizes converge in probability to zero
as $K\to\infty$, it follows from  the martingale central limit theorem
stated in \citet[Chapter~VIII, Theorem~3.11]{jacod-shiryaev} that 
$\big(W'_{Kt}+W''_{Kt}\big)_t$ converges in distribution to a
Brownian motion with variance $q(1)m_2(1)$, as $K\to\infty$. In addition, the same result gives the convergence in probability of the quadratic variation process to $(m_2(1)q(1)t)_t$.

 It remains to show that the last term $W'''$ vanishes. By
definition of the compensated Poisson measure, $W'''$ is the
difference between a jump process driven by $\mathcal{Q}_2$  and a finite
variation part:
 \begin{multline*}
    W'''_t\ =\ 
 \int_0^t \int_{\R_+} \int_0^1 \frac{L_1(Z^K_{s^-},\xi)-1}{K} \1\big(L_1(Z^K_{s^-},\xi)> K \eta_K \big)\1\big(\varrho \leqslant k q({Z}^K_{s^-}) m({Z}^K_{s^-})\big)  \mathcal{Q}_2( ds d\varrho d\xi)\\
-\int_0^t\frac{k}{K}q({Z}^K_{s}) m({Z}^K_{s})\E\Big[\big(L_1(Z^K_{s})-1\big) \1\big(L_1(Z^K_{s})> K \eta_K \big)\,  \big\vert\,  Z_{s}\Big]  ds.
 \end{multline*}

It is direct that with high probability we do not see any jumps in $W'''$, as otherwise the density would escape $\mathscr{B}_{\gamma}$, hence the first line in the above vanishes. 

Besides, since $m(z) \E\big[ L_1(z)\1(L(z)> K \eta_K)\big] \leqslant
\E\big[ L^2(z)\1(L(z)> K \eta_K)\big]$, by \eqref{eq:UI_moment} the
integrand of the finite variation part vanishes uniformly as $K\to\infty$. As a result, the process $\big(W'''_{Kt}\big)_t$ converges in probability to zero. The convergence in distribution under $\mathbb{P}^k_{Z_0}$ for every $k\geqslant 0$ stated in the lemma eventually follows from the Slutsky lemma.
\end{proof}

\subsection{Estimation of the Feynman--Kac term}

This section is dedicated to establishing the following result for the
so-called base case \eqref{eq:base_case} of the moment induction formula.

\begin{prop}\label{prop:estimation_deltak_multitype}
Let $k\geqslant 0$, $t>0$ and let $\psi \colon \R_+ \to \R$ be bounded and
continuous at 1. Then, for any sequence of initial population sizes $(Z_0)_K$ satisfying
$Z_0/K\to 1$ as $K\to\infty$,
\begin{equation*}
    \lim_{K\to\infty} \e^{\beta Z^K_0}\E^k_{Z_0}\bigg[\psi\big( Z^K_{Kt}\big) \e^{-\beta Z^K_{Kt}} \exp\Big(\int_0^{Kt}  k q({Z}^K_s)({m}({Z}^K_s)-1)ds\Big)\bigg] = \psi(1)\exp\Big(-q(1)m_2(1)\mbinom{k}{2}t\Big).
\end{equation*}

\end{prop}

\begin{proof}

First, by Proposition~\ref{prop:coupling_frozen} and the continuity of
$\psi$ at $z=1$, $\psi( Z^K_{Kt})\e^{\beta(Z^K_0 -Z^K_{Kt})} \to \psi(1)$
as $K \to \infty$ in distribution under $\mathbb{P}^k_{Z_0}$.

For the remaining term in the expectation, recall from
\eqref{eq:SDE_under_Q} that, under $\mathbb{P}^k_{Z_0}$, the density process
$({Z}^K_t)$ is solution to
\begin{equation*}
    d {Z}^K_s\ = \  \big({m}({Z}^K_s)-1\big)q({Z}^K_s)\big({Z}^K_s-k/K\big)ds + {m}_2({Z}^K_s)q({Z}^K_s)\frac{k}{K}ds + dW^{(k)}_s,
\end{equation*}
where $W^{(k)}_s$ is the stochastic integral defined in 
\eqref{eq:martingale_before_cuting_in_three_parts}.
Applying Itô's formula to the logarithm function (see \eg
\citet[Chapter II, Theorem 32]{protter_stochastic_2005}) we have
\begin{multline*}
    d\big(\log {Z}^K_s\big)\ =\ \big({m}({Z}^K_s)-1\big)q({Z}^K_s) \Big(1-\frac{k}{K {Z}^K_s}\Big)ds + {m}_2({Z}^K_s)q({Z}^K_s)\frac{k}{K{Z}^K_s}ds\\
   + \frac{dW^{(k)}_s }{{Z}^K_{s^-}}+ \bigg[\log\Big(1+\frac{\Delta {Z}^K_{s}}{{Z}^K_{s^-}}\Big)-\Big(\frac{\Delta {Z}^K_{s}}{{Z}^K_{s^-}}\Big)\bigg].
\end{multline*}
Then, integrating until time $Kt$ and rearranging the terms we obtain that
\begin{multline*}
    \int_0^{Kt} q({Z}^K_s)\big({m}({Z}^K_s)-1\big)ds\ =\   \log \Big(\frac{{Z}^K_{Kt}}{{Z}^K_0}\Big) +  \int_0^{Kt} ({m}({Z}^K_s)-1)q({Z}^K_s)\frac{ k }{K {Z}^K_s }ds\\
      - \int_0^{Kt} {m}_2({Z}^K_s)q({Z}^K_s)\frac{ k }{K {Z}^K_s }ds -  \int_0^{Kt}\frac{1}{{Z}^K_{s^-}} dW^{(k)}_s\\ + \sum_{s\leqslant Kt} \bigg[\Big(\frac{\Delta {Z}^K_{s}}{{Z}^K_{s^-}}\Big)-\log\Big(1+\frac{\Delta {Z}^K_{s}}{{Z}^K_{s^-}}\Big)\bigg].
\end{multline*} 
Define the remainder $\varepsilon(Kt)$ such that
\begin{equation}\label{eq:rest_bias}
    \int_0^{Kt} q({Z}^K_s)({m}({Z}^K_s)-1)ds\ = \ - k q(1)m_2(1) t - W^{(k)}_{Kt} + \frac{1}{2} \big[W^{(k)},W^{(k)}\big]_{Kt} + \varepsilon(Kt).
\end{equation}
We will show subsequently in Lemma~\ref{lem:rest_bias} that the remainder
satisfies $\varepsilon(Kt)\to 0$ in probability as $K\to\infty$, and we assume
that this holds within the current proof. 

By Lemma~\ref{lem:cv_martingale_jacod_shiryaev}, the process
$\big([W^{(k)},W^{(k)}]_{Kt}\big)_{t\geqslant 0}$ converges in probability to
$\big(q(1)m_2(1)t\big)_{t\geqslant 0}$ and $\big(W^{(k)}_{Kt}\big)_{t \geqslant 0}$ converges to
a Brownian motion with variance $q(1)m_2(1)$. Therefore, further using
that $\varepsilon(Kt) \to 0$ by Lemma~\ref{lem:rest_bias},
\[
    \int_0^{Kt} q({Z}^K_s)({m}({Z}^K_s)-1)ds\ 
    \longrightarrow\ 
    \sqrt{q(1)m_2(1)}B_t - q(1)m_2(1)\frac{2k-1}{2}t,
\] 
in distribution under $\mathbb{P}^k_{Z_0}$, as $K \to \infty$. 
Multiplying this expression by $k$ and taking an exponential, we obtain
furthermore that 
\begin{equation}\label{eq:cv_bias}
    \exp\Big(\int_0^{Kt} kq({Z}^K_s)\big({m}({Z}^K_s)-1\big)ds\Big)\ \underset{K\to\infty}{\overset{(d)}{\longrightarrow}}\ \exp\Big(-q(1)m_2(1)\mbinom{k}{2}t\Big)\, \mathcal{E}_t,
\end{equation}
where
\[
\mathcal{E}_t\ =\ \exp\Big(k\sqrt{q(1)m_2(1)}B_{t} -\frac{k^2}{2} 
q(1)m_2(1)t\Big)
\]
is the exponential martingale of the  Brownian motion with variance $k^2
q(1)m_2(1)$, and the convergence in distribution should as before be
understood as under $\mathbb{P}^k_{Z_0}$.

In order to obtain convergence of the expected value corresponding to \eqref{eq:base_case}, we need some additional uniform integrability. The test function $\psi$ being bounded, it suffices to treat the case $\psi=1$. Let $M>0$. Reverting the change of measure and using the exchangeability of the initial particles leads to
\begin{align*}  
\E^k_{Z_0}\bigg[\e^{-\beta Z^K_{Kt}}&\exp\Big(\int_0^{Kt} kq({Z}^K_s)({m}({Z}^K_s)-1)ds\Big)\ \1\Big(\exp\Big(\int_0^{Kt} kq({Z}^K_s)({m}({Z}^K_s)-1)ds\Big)>M\Big)\bigg]\\
 & =\ \E_{Z_0}\bigg[\e^{-\beta Z^K_{Kt}}\prod_{i =1}^{k}Z^{(i)}_{Kt}\cdot \1\Big(\exp\Big(\int_0^{Kt} kq({Z}^K_s)({m}({Z}^K_s)-1)ds\Big)>M\Big)\bigg]\\
  =\ &\frac{1}{(Z_0)_{k}}\E_{Z_0}\bigg[\e^{-\beta
 Z^K_{Kt}}\sum_{\substack{\bm u: u_1\neq\cdots\neq u_k\\ \bm u\in \mathcal{N}_0}} \prod_{i =1}^{k}Z^{(u_i)}_{Kt} \1\Big(\exp\Big(\int_0^{Kt} kq({Z}^K_s)({m}({Z}^K_s)-1)ds\Big)>M\Big)\bigg]\\
&\leqslant\ \frac{1}{(Z^K_0-k/K)^{k}}\E_{Z_0}\bigg[\e^{-\beta Z^K_{Kt}}(Z^K_{Kt})^k\ \1\Big(\exp\Big(\int_0^{Kt} kq({Z}^K_s)({m}({Z}^K_s)-1)ds\Big)>M\Big)\bigg]\\
& \leqslant\ \e^{-1}\Big(\frac{\beta}{Z^K_0-k/K}\Big)^{k} \mathbb{P}_{Z_0}\bigg(\exp\Big(\int_0^{Kt} kq({Z}^K_s)({m}({Z}^K_s)-1)ds\Big) >M\bigg).
\end{align*}
The convergence in distribution of the random variable
$\exp\big(\int_0^{Kt} kq({Z}^K_s)({m}({Z}^K_s)-1)ds\big)$ under
$\mathbb{P}=\mathbb{P}^0$ derived above implies its tightness. Hence, the upper
bound vanishes as $M\to\infty$, uniformly in $K$.
\end{proof}

\begin{lem}\label{lem:rest_bias}
   For all $k\geqslant 0$, for all sequence of initial conditions $(Z_0)_K$ satisfying $Z_0/K\to 1$ and $T>0$, the remainder term $\sup_{s\leqslant T}\varepsilon(Kt)$ introduced in \eqref{eq:rest_bias} converges in probability to zero under $\mathbb{P}^k_{Z_0}$, as $K\to\infty$.
\end{lem}

\begin{proof}
From the definition of $\varepsilon(Kt)$ we have
\begin{multline*}
    \big|\varepsilon(Kt)\big|\ \leqslant \ \underbrace{  \Big| \log \Big(\frac{{Z}^K_{Kt}}{{Z}^K_0}\Big)\Big|}_{A} \ +\ \underbrace{\frac{1}{K} \int_0^{Kt} \Big|({m}({Z}^K_s)-1)q({Z}^K_s)\frac{ k }{{Z}^K_s }\Big|}_{B}ds\\
      +\ \underbrace{\int_0^{Kt}\frac{k }{K}\Big|\frac{{m}_2({Z}^K_s)q({Z}^K_s)}{{Z}^K_s}-q(1)m_2(1)\Big|}_{C} ds\ +\ \underbrace{\Big|\int_0^{Kt}\Big[1-\frac{1}{{Z}^K_{s^-}}\Big] dW^{(k)}_s\Big|}_{D}\\
     +\ \underbrace{\Big|\sum_{s\leqslant Kt} \log\Big(1+\frac{\Delta {Z}^K_{s}}{{Z}^K_{s^-}}\Big)-\Big(\frac{\Delta {Z}^K_{s}}{{Z}^K_{s^-}}\Big)+\frac{1}{2}\big(\Delta {Z}^K_{s}\big)^2\Big|}_{E}.
\end{multline*}

Using the Skorohod representation theorem, we assume as before that the
convergence $Z^K_{tK}\to 1$ uniformly on $[0,T]$ holds almost-surely.
Then, by continuity of $m$, $m_2$ and $q$ at $1$ we directly have that
$A$, $B$ and $C$ converge almost-surely to $0$.
A Taylor series expansion of the logarithm to the second order yields
that, for $t\leqslant T$,
\begin{align}\nonumber
    E\ &= \ \Big|\sum_{s\leqslant Kt} \big(\Delta {Z}^K_{s}\big)^2\bigg[\frac{1}{2}\Big(1-\frac{1}{({Z}^K_{s^-})^2}\Big)\ +\ \frac{1}{({Z}^K_{s^-})^2}  R_{\log}\Big(\frac{\Delta {Z}^K_{s}}{{Z}^K_{s^-}}\Big)\bigg]\Big|\\\nonumber
    &\leqslant \ \sum_{s\leqslant Kt} \big(\Delta {Z}^K_{s}\big)^2\bigg(\frac{1}{2}\Big|1-\frac{1}{({Z}^K_{s^-})^2}\Big|\ +\ \frac{1}{({Z}^K_{s^-})^2}\Big| R_{\log}\Big(\frac{\Delta {Z}^K_{s}}{{Z}^K_{s^-}}\Big)\Big|\bigg)\\ \label{eq:control_E}
      &\leqslant\ \big[W^{(k)},W^{(k)}\big]_{KT}\bigg(\frac{1}{2}\sup_{s\leqslant KT}\ \big|1-(Z^K_{s^-})^{-2}\big|\ +\ \sup_{s\leqslant KT}\ \frac{1}{({Z}^K_{s^-})^2}\Big| R_{\log}\Big(\frac{\Delta {Z}^K_{s}}{{Z}^K_{s^-}} \Big)\Big|\bigg),
\end{align}
where the remainder $R_{\log}:x\mapsto x^{-2}(\log(1+x)-x+x^2/2)$ satisfies $\lim_{x\to 0}R_{\log}(x)=0$. By Lemma~\ref{lem:cv_martingale_jacod_shiryaev}, the quadratic variation $[W^{(k)},W^{(k)}]_{KT}$ converges in probability to a finite limit. Besides, by Proposition~\ref{prop:coupling_frozen} the population size $Z^K_{Ks}$ converges  almost-surely to $1$, uniformly on $[0,T]$ and there exists $\gamma>0$ such that the exit time  from \eqref{eq:stopping_times} satisfies $(T_{\gamma}\wedge KT)/K \to  T$ almost-surely. Re-using the argument, this uniform convergence implies that with high probability, the supremum of the jumps fulfills
\[
\sup_{s < Kt} |\Delta Z^K_s| \ \to \ 0.
\]
Therefore, by Slutsky's theorem the product \eqref{eq:control_E} converges to zero converges to $0$ in probability.

To derive the result, it remains to control $D$. In the definition of $D$,  the stochastic integral within the absolute values is a local martingale. When stopped at $T_{\gamma}$, its quadratic variation satisfies
\begin{align*}
   \big[D,D\big]_{Kt\wedge T_{\gamma}}\ &=\ \sum_{s\leqslant Kt \wedge T_{\gamma}} \big(1-(Z^K_s)^{-1}\big)^2\big|\Delta W^{(k)}_s\big|^2\\
   &\leqslant\    \big[ W^{(k)},W^{(k)}\big]_{KT\wedge T_{\gamma}}\sup_{s\leqslant KT \wedge T_{\gamma}}\big(1-(Z^K_{s})^{-1}\big)^2,
\end{align*}
which by Lemma~\ref{lem:cv_martingale_jacod_shiryaev} and convergence of the density converges in probability to zero. Moreover, its jumps are uniformly bounded by $(1+1/(1-\gamma))$. Thus, re-applying \citet[Chapter VIII, Theorem 3.11]{jacod-shiryaev} and Proposition~\ref{prop:coupling_frozen}, $D$ converges in probability to zero as $K\to\infty$.
Putting everything together, we finally obtain that for all $k\geqslant 0$, $\sup_{t\leqslant T}\varepsilon(Kt)$ converges to zero in probability under $\mathbb{P}^k_{Z_0}$.
\end{proof}

\section{Convergence of the genealogy}\label{sec:convergence}

In this final section, we begin by showing in
Proposition~\ref{prop:cv_moments_planaires} that the penalized planar
moments measures \eqref{eq:moments_planaires} admit a limit.
Moreover, we will see that this limit is a probability distribution
supported on the planar chronological forests with binary mergers which corresponds to the distribution of a Markov process. In Proposition~\ref{prop:deplanarization}, we
then identify this limiting distribution as that of the planarized Kingman
coalescent. Finally, we complete the proof of the main theorem by linking
the genealogies of a uniform sample from the interacting
population model of Section~\ref{sec:model} to the penalized planar
moments.

\subsection{Limit of the penalized moments}

For any planar coalescent process $\Pi\in \mathbb{G}^{k}$, we define
$\Pi/K\in \mathbb{G}^{k}$ to be the planar coalescent on $\{1,\ldots, k\}$ obtained after rescaling time
by $1/K$.

\begin{prop}[Limit of the moments]\label{prop:cv_moments_planaires} 
    Let $k\geqslant 1$, $\beta>0$ and $t>0$. Let $\varphi : \mathbb{G}^k \to \R$ 
    and $\psi : \mathbb{R}_+ \to \mathbb{R}$ be two bounded functionals with $\psi$ continuous at $1$.
    There exists a probability measure $\mathbf{\bar{M}}^{k,t}$ supported on planar $k$-coalescent with binary mergers such
    that if $(Z_0)_K$ is a deterministic sequence of initial conditions satisfying $Z_0/K\to 1$
    as $K\to\infty$, then 
    \begin{equation}\label{eq:def_limit_measure}
        \lim_{K\to\infty} \mathbf{M}^{k,Kt}_{Z_0}\big(\varphi(\cdot/K),\ \psi\big) 
        \ = \   
        \psi(1)\mathbf{\bar M}^{k,t}(\varphi). 
    \end{equation}
    Furthermore, let $k\geqslant 2$ and $\varphi$ be of the product form defined in
    \eqref{eq:product_form_functionals}, with $d=2$. That is  
    \begin{equation*}
    \forall \Pi\in\mathbb{G}^{k}, \ \ 
    \varphi(\Pi)\ \coloneqq\ \1\big(\tau_1<t,\  c_1=(i,2)\big)  \varphi_1(\tau_1)  \varphi_2\big(\Theta_{\tau_1}(\Pi)\big),
    \end{equation*}
    where $i\in\{1,\ldots, k-1\}$ and $\varphi_1,\, \varphi_2$ are bounded measurable functions.
    Then the following holds
    \begin{equation}\label{eq:limit_recurrence}
    \begin{dcases}
        \mathbf{\bar{M}}^{k,t}(\tau_1>t)\ =\ 
        \exp\Big(-q(1)m_2(1)\mbinom{k}{2}t\Big) \\
        \mathbf{\bar M}^{k,t}(\varphi)\ =\ \int_{0}^{t}\varphi_1(s)
        \mathbf{\bar M}^{k-1, s}(\varphi_2) \frac{q(1)m_2(1)}{k-1}\binom{k}{2}
        \exp\Big(-q(1)m_2(1)\mbinom{k}{2}s\Big) d s.
    \end{dcases}
    \end{equation}
\end{prop}
We start with a lemma before proceeding to the proof. Recall the notation $\mathcal{M}^{(k,d),t}_z$ from
\eqref{eq:psi_tilde_recursive_formula}.

\begin{lem}\label{cor:cv_m_cal_psi}
Let $\psi$ be bounded, nonnegative and continuous at 1, $k\geqslant 1$ and 
$t>0$. For any sequence $(z_K)_K$ such that $z_K\to 1$ as $K\to\infty$, 
\begin{equation*}
    \lim_{K\to\infty}{\cal M}^{(k,2),Kt}_{z_K}(\psi)\ =\ \psi(1)q(1)m_2(1)\exp\Big(-q(1)m_2(1)\mbinom{k}{2}t\Big).
\end{equation*}
\end{lem}

\begin{proof}
    First, for any $A > 0$
    \begin{align*}
        \big|m^{\beta/K}_2(z_K) - m_2(z_K)\big| 
      \  & =\
        \E\big[ \big(L(z_K)\big)_2 \big( 1 - e^{-\beta L(z_K) / K}\big) \big] \\
        & \leqslant\
        \sup_{z \in (1-\varepsilon, 1+\varepsilon)} 
        \Big( \E\big[ \big(L(z_K)\big)_2 \1_{\{L(z_K) > A\}} \big] \Big)
        +
        A^2 \E\big[ 1- e^{-\beta L(z_K) / K} \big].
    \end{align*}
    The right-hand side vanishes by letting $K \to \infty$ first (and
    noting that $\big(L(z_K)\big)_K$ is tight, since it has uniformly bounded
    second moment) and then $A \to \infty$ using \eqref{eq:UI_moment}.
    Therefore, since $m_2(z)$ is continuous at $z=1$, we deduce
    that 
    \[
        \lim_{K \to \infty} m_{2, \beta/K}(z_K)\ =\ m_2(1).
    \]
    Similarly,
    \[
        \mathbb{P}\big(L_{2}^{\beta/K}(z_K) \geqslant A\big) \ \leqslant\
        \frac{1}{m^{\beta/K}_2(z_K)} \E\Big[ \big(L(z_K)\big)_2 \1_{\{ L(z_K) \geqslant
        A\}} \Big].
    \]
    The right-hand side vanishes uniformly in $K$ as $A \to \infty$ by
    \eqref{eq:UI_moment} and thus $\big(L^{\beta/K}_2(z_K)\big)_K$ is tight and 
    \[
        z_K + \frac{L^{\beta/K}_2(z_K)-1}{K}
      \  \underset{K\to\infty}{\longrightarrow}\ 1
    \]
    in probability. We can therefore apply
    Proposition~\ref{prop:estimation_deltak_multitype}, yielding that 
    \begin{equation*}
        \lim_{K\to\infty} q(z_K) m^{\beta/K}_2(z_K) \e^{\beta/K}
        \mathbf{M}^{k,Kt}_{K z_K+L^{\beta/K}_2(z_K)-1}(\psi) 
        \ =\ \psi(1)q(1)m_2(1)\exp\Big(-q(1)m_2(1)\mbinom{k}{2}t\Big)
    \end{equation*}
    holds in probability. 
    
    To complete the proof, we start by noting that, for $K$ large
    enough we have $z_K \in (1/2, 2)$ and that by the definition of the base case moment in \eqref{eq:moments_planaires}, the following upper bound holds:
    \[
     \mathbf{M}^{k,Kt}_{K z_K}(\psi)\ \leqslant\ \frac{\lVert\psi\rVert
        k!\, \sup_{z > 0} (\e^{-\beta z} z^k)}
        {\inf_{z \in (1/2,2)} (\e^{-\beta z} z^k)}.
    \]
    That is, using in addition that for all $z$ we have $\E\big[\e^{\beta L_2^{\beta/K}(z)/K}\big]= m_2(z)/ m^{\beta/K}_2(z)<\infty$, the random variable $\mathbf{M}^{k,Kt}_{K z_K+L^{\beta/K}_2(z_K)-1}(\psi)$ is uniformly integrable.
\end{proof}

\begin{proof}[Proof of Proposition~\ref{prop:cv_moments_planaires}]
If $\varphi = \1(\tau_1 > t)$ is the indicator that no coalescence events
occur before time $t$, by Theorem~\ref{thm:recursion_moment}
\[
    \mathbf{M}_{Z_0}^{k,Kt}\big(\varphi(\cdot/K), \psi\big)\ = \
    \e^{\beta Z^K_0}\E^k_{Z_0}\bigg[ 
        \psi(Z^K_{Kt})\e^{-\beta Z^K_{Kt}}\exp\Big(k\int_0^{Kt} q(Z^{K}_s) 
        \big(m(Z^K_s)-1\big) ds \Big)
    \bigg].
\]
Proposition~\ref{prop:estimation_deltak_multitype} entails that
\begin{equation} \label{eq:limit_base_case}
    \lim_{K \to \infty} \mathbf{M}_{Z_0}^{k,Kt}\big(\varphi(\cdot/K), \psi\big) 
  \  =\ \psi(1) \exp\Big(-q(1)m_2(1)\mbinom{k}{2}t\Big)
\end{equation}
which is a reformulation of \eqref{eq:def_limit_measure}. Since there are
no coalescence events in a tree with a single leaf, this also proves the
result for $k=1$.

Now, let $\varphi$ be of the product form \eqref{eq:product_form_functionals}.
We will prove the result by induction and assume that, for all $p < k$, 
\eqref{eq:def_limit_measure} holds and that the limiting probability
measure $\mathbf{\bar{M}}^{p,t}$ is supported on binary genealogies. Rescaling time by $K$ in
Theorem~\ref{thm:recursion_moment} shows that 
\begin{multline}\label{eq:rec_before_limit}
    \mathbf{M}^{k,Kt}_{Z_0}\big(\varphi( \cdot/K),\psi\big)
    \ =\ 
    \frac{1}{k-d+1}\binom{k}{d}\int_0^{t} \varphi_1(t-s)\frac{K}{(Z_0-k+d-1)_{d-1}}\\
    \times\mathbf{M}^{k-d+1,Ks}_{Z_0}\Big(
        {\varphi}_2(\cdot/K),\ \mathcal{M}^{(k,d),K(t-s)}_{\cdot}(\psi)
    \Big) ds.
\end{multline}
We will prove the result by letting $K \to \infty$ in this identity.
This requires sufficient integrability, which is the delicate step of
the proof.

For $d = 2$ and after re-scaling time by $K$, $\mathbf{M}^{k-1,Ks}_{Z_0}$ is a measure on
$\mathbb{G}^{k-1} \times \R_+$. By our induction hypothesis, its second
marginal (on $\R_+$) converges weakly to a Dirac mass at $z=1$.
Lemma~\ref{cor:cv_m_cal_psi} and an adaptation to finite measures of
the continuous mapping theorem from \citet[Theorem
5.27]{kallenberg_foundations_2021} ensure that the pushforward of
$\mathbf{M}^{k-1,Ks}_{Z_0}$ by the map 
\[
    (\Pi,\, z)\ \longmapsto\ \mathcal{M}_z^{(k,2),K(t-s)}(\psi)
\]
converges weakly to a Dirac mass at $\psi(1)q(1)m_2(1) \exp(
-q(1)m_2(1)\binom{k}{2}(t-s))$. Therefore, by Fatou's lemma we obtain that
\begin{equation} \label{eq:fatou_1}
    \liminf_{K\to\infty}\ \mathbf{M}^{k-1,Ks}_{Z_0}\Big(
    1,\ \mathcal{M}^{(k,2),K(t-s)}_{\cdot}(1)\Big) 
    \ \geqslant\ 
    q(1)m_2(1)\exp\Big(-q(1)m_2(1)\mbinom{k}{2}(t-s)\Big).
\end{equation}
Let us now denote by 
\[ 
    \varphi_{(i,d)} : \Pi\ \longmapsto\ \1\big(\tau_1<Kt,\ c_1=
    (i,d)\big),\quad 2\leqslant d\leqslant k,\ i\leqslant  k-d+1
\]
the indicator that at least one coalescence event occurs in the genealogy
and that the last one corresponds to merging the adjacent lineages
labeled from $i$ to $i+d-1$ together, in the notation of
Section~\ref{sec:marked_forests}. 

Decomposing according to the last
merger, 
\begin{align*}
    \mathbf{M}^{k,Kt}_{Z_0}(1,1)\ 
    &=\ \sum_{d=2}^k\sum_{i=1}^{k-d+1} 
    \mathbf{M}^{k,Kt}_{Z_0}(\varphi_{(i,d)},1)
    + \mathbf{M}^{k,Kt}_{Z_0}(1) \\
    &\geqslant\ 
    \sum_{i=1}^{k-1} 
    \mathbf{M}^{k,Kt}_{Z_0}(\varphi_{(i,2)},1)
    + \mathbf{M}^{k,Kt}_{Z_0}(1).
\end{align*}
Combining this inequality with \eqref{eq:limit_base_case}, \eqref{eq:rec_before_limit}  and 
\eqref{eq:fatou_1}, we obtain
\begin{align*}
   \liminf_{K \to \infty}\ 
   & \sum_{i=1}^{k-1} 
    \mathbf{M}^{k,Kt}_{Z_0}(\varphi_{(i,2)},1)
    + \mathbf{M}^{k,Kt}_{Z_0}(1)\\
   & \geqslant\  
        \frac{1}{k-1} \binom{k}{2} 
        \sum_{i=1}^{k-1} \int_0^t 
        q(1)m_2(1)\exp\Big(-q(1)m_2(1)\mbinom{k}{2}(t-s)\Big) ds 
        + \exp\Big(-q(1)m_2(1)\mbinom{k}{2}t\Big) \\
    &=\ \Big( 1-\exp\Big(-q(1)m_2(1)\mbinom{k}{2}t\Big) \Big)
        + \exp\Big(-q(1)m_2(1)\mbinom{k}{2}t\Big)  \\
    &=\ 1.
\end{align*}
On the other hand, since $Z_{Kt}^K \to 1$ by
Proposition~\ref{prop:coupling_frozen}, the continuous mapping
theorem shows that 
\[
   \mathbf{M}^{k,Kt}_{Z_0}(1,1)
  \ =\ 
  \frac{1}{(Z_0)_k} \E_{Z_0}\Big[\e^{\beta(Z^K_0-Z^K_{Kt})} (Z_{Kt})_k \Big]
   \underset{K\to\infty}{\longrightarrow} 1.
\]
This allows us to deduce two facts. First, for all $d > 2$,
\[
    \lim_{K \to \infty} \mathbf{M}^{k,Kt}_{Z_0}(\varphi_{(i,d)},1)\ =\ 0,
\]
so that any possible limit of $(\mathbf{M}^{k,Kt}_{Z_0})_K$ is supported
on binary genealogies. Second, for $d=2$,
\[
    \lim_{K \to \infty}\ \frac{1}{k-1} \binom{k}{2} \int_0^t 
    \mathbf{M}^{k-1,Ks}_{Z_0}\Big(1,\ \mathcal{M}^{(k,2),K(t-s)}_{\cdot}(1)\Big) ds
    \ =\ \frac{1}{k-1}\big( 1 - \e^{-q(1)m_2(1) \binom{k}{2} t} \big).
\]
From there, the Vitali convergence theorem (see 
\citet[Lemma~22.7]{schilling_measures_2017} with $p=1$) gives that the pushforward
of $ds \otimes \mathbf{M}^{k-1,Ks}_{Z_0}$ by
\[
    (s, \Pi, z)\ \longmapsto\ \mathcal{M}^{(k,2), K(t-s)}_z(1)
\]
is a uniformly integrable collection of measures (on $\R_+$). In turn,
letting $K \to \infty$ in \eqref{eq:rec_before_limit} and using the
dominated convergence theorem, for any continuous bounded $\varphi_1$,
$\varphi_2$, and $\psi$ we have
\begin{multline*}
    \lim_{K \to \infty}\ \frac{1}{k-1} \binom{k}{2} \int_0^t \varphi_1(t-s)
    \mathbf{M}^{k-1,Ks}_{Z_0}\Big(\varphi_2(\cdot/K),\, 
    \mathcal{M}^{(k,2),K(t-s)}_{\cdot}(\psi)\Big) ds \\
   \ =\ \frac{\psi(1) q(1) m_2(1)}{k-1} \binom{k}{2}
    \int_0^t \varphi_1(t-s) \mathbf{\bar{M}}^{k-1,t-s}(\varphi_2)
    \e^{-q(1)m_2(1)\binom{k}{2}(t-s)} ds.
\end{multline*}
This completes the induction, and the proof.
\end{proof}

\subsection{Completing the proof}

We define the planar $k$-Kingman coalescent with rate $r>0$ as the
process on planar partitions of $\{1,\ldots, k\}$ such that the process starts from
the partition into singletons and, when the process has $p$ blocks, any
consecutive pair of blocks coalesces at rate
 \[
\frac{r}{p-1} \binom{p}{2}.
\] 
Note that, since there are $p-1$ consecutive pairs of
adjacent blocks among $p$ ordered blocks, the total coalescent rate is
consistent with that of the (plain) Kingman coalescent with rate
$r$. In fact, we have the following result which corresponds to \citet[Lemma~3.1]{Foutel-Rodier:UMS}.

\begin{lem}{\cite{Foutel-Rodier:UMS}.}\label{lem:felix}
    Let $(\widetilde{\Pi}^k_t, t\geqslant 0)$ be the planar $k$-Kingman
    coalescent with rate $r$. Let $\sigma$ be an independent random
    uniform permutation of $\{1,\ldots, k\}$. Then $(\sigma(\widetilde{\Pi}^k_t))_{t \geqslant 0}$
    is identical in law to the (plain) $k$-Kingman coalescent with
    rate $r$. Here, $\sigma(\widetilde{\Pi}_t)$ is the partition
    obtained by permuting the indices of $\widetilde\Pi_t$ by $\sigma$.
\end{lem}
The limit object $\mathbf{\bar M}^{k,t}$ was introduced as a probability measure on planar $k$-coalescent processes, described by a sequence of coalescent times and events.
\begin{prop}\label{prop:deplanarization}
    The measure $\mathbf{\bar M}^{k,t}$ is equal to the distribution of
    the planar $k$-Kingman coalescent with rate $m_2(1) q(1)$, restricted
    to $[0, t]$. 
\end{prop}

\begin{proof}
    Looking at \eqref{eq:limit_recurrence} in 
    Proposition~\ref{prop:cv_moments_planaires}, we see that the first
    jump of the coalescent under $\mathbf{\bar{M}}^{k,t}$ occurs after an
    exponential time with rate $q(1)m_2(1) \mbinom{k}{2}$. Conditional on
    that jump time, each of the $k-1$ pairs of consecutive blocks
    coalesces with equal probability. This is the description of the
    transitions of the planar $k$-Kingman coalescent.
\end{proof}

\begin{proof}[Proof of Theorem~\ref{thm:main}]
The convergence of the population size stated in the theorem 
is a consequence of Proposition~\ref{prop:coupling_frozen}.

For the second point, let $t > 0$. By Proposition~\ref{prop:coupling_frozen}, one can find $\gamma>0$ such that $\mathbb{P}(Z^K_{Kt} \in \mathscr{B}_\gamma) \to 1$ as $K \to \infty$. Ergo,
\begin{align*}
    \E_{Z_0}\Big[\varphi(\Pi^{k,K}) \, \big\vert\,Z_{Kt}\geqslant k\Big] \ 
    &= \ \E_{Z_0}\Big[\varphi(\Pi^{k,K}) \1(Z^K_{Kt} \in 
    \mathscr{B}_\gamma)\Big] + o(1) \\
    &= \ \E_{Z_0}\bigg[ 
        \frac{\1(Z^K_{Kt} \in \mathscr{B}_\gamma)}{\big(Z_{Kt}\big)_{k}} 
        \sum_{\substack{\bm v : v_1<\cdots<v_{k} \\ \bm v \in{\cal N}_{Kt}}} 
        \sum_{\sigma\in\mathfrak{S}_k} \varphi\circ\sigma(\Pi^{{\bm v}}/K)
    \bigg] + o(1) \\
    &= \ \frac{1}{K^k} \E_{Z_0}\bigg[ 
        \frac{\1(Z^K_{Kt} \in \mathscr{B}_\gamma)}{\big(Z^K_{Kt}\big)^k} 
        \sum_{\substack{\bm v : v_1<\cdots<v_{k} \\ \bm v \in{\cal N}_{Kt}}} 
        \sum_{\sigma\in\mathfrak{S}_k} \varphi\circ\sigma(\Pi^{{\bm v}}/K)
    \bigg] + o(1),
\end{align*}
where $\mathfrak{S}_k$ is the set of permutations of $\{1,\ldots, k\}$. Introducing 
\[ 
    \psi_{\gamma}:z\ \longmapsto\ \frac{\e^{\beta z}}{z^k} \1\big(z\in
    \mathscr{B}_{\gamma}\big),
\]
we can rewrite
\begin{multline*}
    \frac{1}{K^k} \E_{Z_0}\bigg[ 
        \frac{\1(Z^K_{Kt} \in \mathscr{B}_\gamma)}{\big(Z^K_{Kt}\big)^k} 
        \sum_{\substack{\bm v : v_1<\cdots<v_{k} \\ \bm v \in{\cal N}_{Kt}}} 
        \sum_{\sigma\in\mathfrak{S}_k} \varphi \circ \sigma(\Pi^{{\bm v}}/K)
    \bigg] \\
     =\ \frac{(Z_0)_k}{K^k} \e^{-\beta Z_0^K}
    \mathbf{M}^{k,Kt}_{Z_0}\Big( \frac{1}{k!} \sum_{\sigma \in
    \mathfrak{S}_K} \varphi \circ \sigma,\, \psi_\gamma \Big).
\end{multline*}
By Proposition~\ref{prop:cv_moments_planaires} and the steps above, we finally derive
\[
    \lim_{K \to \infty} 
    \E_{Z_0}\Big[\varphi(\Pi^{k,K}) \, \big\vert\,Z_{Kt}\geqslant k\Big]
   \ =\ \mathbf{\bar{M}}^{k,t}\Big( \frac{1}{k!} \sum_{\sigma \in
    \mathfrak{S}_k} \varphi \circ \sigma \Big),
\]
and the result follows from Lemma~\ref{lem:felix}.

\end{proof}

\phantomsection

\end{document}